\documentclass[reqno]{amsart}
\usepackage{willie_style}

\usepackage{xr-hyper}
\title{Skew Howe duality for Types $\mathbf{BD}$ via  $q$-Clifford algebras}
\author[W.~Aboumrad]{Willie Aboumrad}
\address[W.~Aboumrad]{The Institute for Computational and Mathematical Engineering (ICME) at Stanford University}
\email{willieab@stanford.edu}
\urladdr{https://web.stanford.edu/~willieab}

\date{}

\begin{document}
\maketitle

\begin{abstract}
	We extend a quantized skew Howe duality result for Type $\mathbf{A}$ algebras to orthogonal types via a seesaw. We develop an operator commutant version of the First Fundamental Theorem of invariant theory for $U_q(\mathfrak{so}_n)$ using a double centralizer property inside a quantized Clifford algebra. We obtain a multiplicity-free decomposition of tensor powers of the $U_q(\mathfrak{so}_{2n})$ spin representation by explicitly computing joint highest weights with respect to an action of $U_q(\mathfrak{so}_{2n}) \otimes U_q'(\mathfrak{so}_m)$. Clifford algebras are an essential feature of our work: they provide a unifying framework for classical and quantized skew Howe duality results that can be extended to include orthogonal algebras of types $\mathbf{BD}$. 
\end{abstract}

\section{Introduction}
In this \chorpaper\ we develop skew Howe duality for orthogonal types, both classical and quantum. The results in this \chorpaper\ extend those developed in \crossrefglnch. Thus we recommend the reader to review \crossrefglnch.

The classical situation is as follows. Let $G$ and $H$ be reductive subgroups of a complex orthogonal group, such that $G$ is the centralizer of $H$ and vice versa. We will refer to $G$ and $H$ as a \textit{dual reductive pair}. We will restrict the projective spin representation of $O(2nm)$ to $G \times H$ and decompose it into irreducibles. More precisely, we will construct a \textit{seesaw} \cite[Proposition~38.4]{Bump}, as depicted in the following diagram.
\begin{equation}\label[diag]{classical seesaw}
	\begin{tikzcd}[column sep=-4.5em]
		& S^{\otimes m} \cong \bigwedge(\mathbb{C}^{nm}) \cong \bigwedge(\mathbb{C}^n)^{\otimes m}
		& \\[-20pt]
		& \text{\Large$\circlearrowleft$}
		& \\[-20pt] 
		& O(2nm)
		& \\[-5pt]
		O(2n) \ar[ur, no head] \ar[d, no head] \ar[ddrr, no head, dashed]	
		& 
		&  O(2m) \ar[ul, no head] \ar[d, no head] \ar[ddll, no head, dashed] \\
		GL(n) \ar[rr, no head, dashed] \ar[d, no head]
		&
		& GL(m) \ar[d, no head] \\
		SO(n) 
		&
		& SO(m)
	\end{tikzcd}
\end{equation}
Here $\bigwedge(\mathbb{C}^{nm})$ is the $O(2nm)$, or rather the $\mathrm{Pin}(2nm)$, spin module. The group $\mathrm{Pin}(2nm) \subset Cl\left(\mathbb{C}^{nm} \otimes (\mathbb{C}^{nm})^*\right)$ is the double (spin) cover of $O(2nm)$. Dashed lines indicate commuting embeddings. In the third row, we need the \textit{full} orthogonal group, instead of the special orthogonal group, in order to ensure a multiplicity-free decomposition of $\bigwedge(\mathbb{C}^{nm})$.

We are mainly interested in the quantum case. Quantum groups are more closely related to enveloping algebras than Lie groups, so we work at the Lie algebra level throughout. We use the commuting embeddings of \crossrefgln{commuting_a_embeddings} to construct explicit homomorphisms as described by the following diagram.
\begin{equation}\label[diag]{classical seesaw}
	\begin{tikzcd}[column sep=-5em]
		& S^{\otimes m} \cong \bigwedge(\mathbb{C}^{nm}) \cong \bigwedge(\mathbb{C}^n)^{\otimes m}
		& \\[-20pt]
		& \text{\Large$\circlearrowleft$}
		& \\[-20pt] 
		& Cl\left((\mathbb{C}^n \otimes \mathbb{C}^m) \oplus (\mathbb{C}^n \otimes \mathbb{C}^m)^*\right) \quad
		& \\[-5pt]
		\mathfrak{o}_{2n} \ar[ur, no head] \ar[d, no head] \ar[ddrr, no head, dashed]	
		& 
		& \mathfrak{o}_{2m} \ar[ul, no head] \ar[d, no head] \ar[ddll, no head, dashed] \\
		\mathfrak{gl}_n \ar[rr, no head, dashed] \ar[d, no head]
		&
		& \mathfrak{gl}_m \ar[d, no head] \\
		\mathfrak{so}_n
		&
		& \mathfrak{so}_m
	\end{tikzcd}
\end{equation}
We may think of the algebra $\mathfrak{o}_{2p}$, defined in \ref{on defn}, as the Lie algebra for the full pin group $\mathrm{Pin}(\mathbb{C}^{2p}) \cong \mathrm{Spin}(\mathbb{C}^{2p}) \rtimes \mathbb{Z}_2$, which is a double-cover of $O(\mathbb{C}^{2p})$. We assume any space of the form $V \oplus V^*$ is equipped with the canonical symmetric bilinear form arising from the duality pairing between $V$ and $V^*$. The exterior algebras are isomorphic as modules of the Clifford algebra $\Cln^{\otimes m} \cong \Clnm$.

The group homomorphisms $GL(p) \hookrightarrow SO(2p)$, for any $p$, which induce the Lie algebra embeddings $\mathfrak{gl}_p \hookrightarrow \mathfrak{so}_{2p}$, are key. They allow us to identify certain $\mathfrak{so}_{2p}$-generators with $\mathfrak{gl}_p$-generators, so we can use our work from \crossrefglnch to prove analogous duality results for orthogonal groups. The induced map $\mathfrak{gl}_p \hookrightarrow \mathfrak{so}_{2p}$ has a quantum analogue embedding $U_q(\mathfrak{gl}_p) \hookrightarrow U_q(\mathfrak{so}_{2p})$, which helps transfer our skew $\Uqgln \otimes \Uqglm$-duality \crossrefgln{uqgln uqglm duality} to a result for Types $\mathbf{BD}$.

The inclusion maps $O(p) \hookrightarrow GL(p)$ are also important. They result from the restriction of $GL(p)$ to a subgroup of invertible isometries and they induce Lie algebra embeddings $\mathfrak{so}_p \hookrightarrow \mathfrak{gl}_p$. These maps allow us to identify $\mathfrak{so}_p$-generators with elements of $\mathfrak{gl}_p$ and base our constructions on the results of \crossrefglnch. However, the embeddings $\mathfrak{so}_p \hookrightarrow \mathfrak{gl}_p$ 
have no quantum analogue $U_q(\mathfrak{so}_p) \hookrightarrow U_q(\mathfrak{gl}_p)$! 

Since there is \textit{no} algebra map $U_q(\mathfrak{so}_{p}) \to U_q(\mathfrak{gl}_p)$, in the quantum case we must instead work with the \textit{non-standard deformation} $\Uqprime$ of $\mathfrak{so}_m$. The non-standard deformation $\Uqprime$ may be understood as a quantization of the compact form of $\mathfrak{so}_m$ that is compatible with the embeddings $U_q'(\mathfrak{so}_m) \supset U_q'(\mathfrak{so}_{m-1}) \supset \cdots \supset U_q'(\mathfrak{so}_3)$ and whose modules can be constructed explicitly using bases indexed by Gelfand-Tsetlin patterns as in the classical case \cite{gavrilik_iorgov_2}. The algebra $\Uqprime$ was first introduced in \cite{gavrilik_klimyk_1991} and its representation theory has been studied in e.g.\cite{iorgov_klimyk_2005,iorgov_klimyk_2000} and most recently in \cite{wenzl_2020}. See also the references therein. Letzter studied the co-ideal algebra structure of $\Uqprime$ in \cite{letzter_2000,letzter_2019}. 

It is a remarkable feature of the quantum case in the orthogonal setting that the commutant of the $\Uqson$ spin action is not generated by a Drinfeld-Jimbo quantum group. Instead, the centralizer is described in terms of $\Uqprime$, which we realize as a co-ideal subalgebra of $\Uqglm$. While \cite{kolb_letzter_2008} obtains some triangular decomposition for $\Uqprime$, this non-standard deformation does not support an analogue of an ``upper triangular'' Borel subalgebra. This means that the method of identifying commuting actions of simple root vectors in terms of Clifford algebra operators, which succeeded in deriving a $\Uqgln \otimes \Uqglm$-duality result via embeddings into a quantized Clifford algebra in \crossrefglnch, does not directly carry over to the orthogonal setting. However, since we realize $\Uqprime$ as a bona fide subalgebra of $\Uqglm$ using explicit formulas in terms of the standard $\Uqglm$ generators, we can still use the results of \crossrefglnch\ to obtain a quantized duality theorem for Types $\mathbf{BD}$.

In \Cref{q orthog duality} we construct embeddings of $\Uqdn$ and $\Uqprime$ into the \textit{quantum Clifford algebra} $Cl_q(nm)$ to develop the seesaw depicted in \Cref{quantum seesaw}. The quantized Clifford algebra was first defined by Hayashi in \cite{hayashi_1990}. In this \chorpaper\ we use a similar version due to Kwon \cite{kwon_2014}. In \crossrefcliffch\ we study a more general version in depth and obtain new results on the algebraic structure and representation theory of quantized Clifford algebras, including a description of the center calculation, a factorization as a tensor product, and a complete list of irreducible representations. 

Much like in the classical case, our embeddings rely crucially on our skew Howe duality result for Type $\mathbf{A}$\longer{,} \crossrefgln{uqgln uqglm duality}. 
\begin{equation}\label[diag]{quantum seesaw}
	\textstyle
	\begin{tikzcd}[column sep=-5em]
		& \bigwedge_q(\Vn)^{\otimes m} \cong S^{\otimes m} \cong \bigwedge_q(V^{(nm)})
		& \\[-20pt]
		& \text{\Large$\circlearrowleft$}
		& \\[-20pt] 
		& Cl_q(nm)
		& \\[-5pt]
		U_q(\mathfrak{o}_{2n}) \ar[ur, no head] \ar[d, no head] \ar[ddrr, no head, dashed]	
		& 
		&  U_q(\mathfrak{o}_{2m}) \ar[ul, no head] \ar[d, no head] \ar[ddll, no head, dashed] \\
		\Uqgln \ar[rr, no head, dashed] \ar[d, no head]
		&
		& \Uqglm \ar[d, no head] \\
		U_q'(\mathfrak{so}_n)
		&
		& U_q'(\mathfrak{so}_m)
	\end{tikzcd}
\end{equation}
In the top row we have isomorphisms of $\Uqdn$-modules. To obtain multiplicity-free decompositions, we extend the $\Uqdn$-action to the \textit{full orthogonal quantum group} $U_q(\mathfrak{o}_{2n})$, recalled in \Cref{uqon defn}. The algebra $U_q(\on)$ is a semidirect product $U_q(\son) \rtimes \mathbb{Z}_2$ that serves as the quantum analogue of $O(n)$.

Braided exterior algebras were first defined in \cite{berenstein} as $\Uqg$-module analogues of the classical exterior algebras. In \crossrefgln{braided_ext_alg} we recall the construction of $\bigwedge_q(\Vn)$, including $Cl_q(n)$-module structure, in detail.

Contrary to the classical situation, no representation of the quantized Clifford algebra $Cl_q(p) \to \End\left(\bigwedge_q(V^{(p)})\right)$ is faithful. This phenomenon explains how we may obtain non-commuting homomorphic images of $U_q(\mathfrak{o}_{2n})$ and $U_q'(\mathfrak{so}_n)$ inside $Cl_q(nm)$ that nevertheless induce commuting actions on the $Cl_q(nm)$-module $\bigwedge_q(V^{(nm)})$.

There are related results in the literature. The dual reductive pair $\left(U_q'(\mathfrak{so}_n), \, U_q(\mathfrak{so}_{2m})\right)$ appears in \cite{sartori_tubbenhauer_2018}. The authors prove their duality result by developing webs and diagrammatical categories. The argument used in \cite{sartori_tubbenhauer_2018} is quite different from ours, which is based on $q$-Clifford algebras. In addition, the results in \cite{sartori_tubbenhauer_2018} do not include a theory for the $\Uqson$-spin module.

However, the recent pre-print \cite{wenzl_spin_duality} indeed discusses a duality result for tensor powers of the quantum spinor module. The author discovered this pre-print while this work was in preparation. While Wenzl does use the quantized Clifford algebra in \cite{wenzl_spin_duality}, his work does not include an explicit joint highest weight vector calculation and it does not comment on the nature of the duality in relation to $\Uqgln \otimes \Uqglm$-duality. In \cite{wenzl_spin_duality}, Wenzl obtains an action of $\Uqprime$ on $S^{\otimes m}$ by quantizing the canonical $O(n)$-invariant element in $\mathbb{C}^n \otimes \mathbb{C}^n$. In contrast, in this work we realize $\Uqprime$ as a co-ideal subalgebra of $\Uqglm$ and use our results from \crossrefglnch. In addition, this work benefits from a new understanding of the representation theory of quantized Clifford algebras, including a calculation of their center, which is developed in \crossrefcliffch.

Various papers deal with the classical situation. For instance, \cite{Heo_2022,gerber_guilhot_lecouvey_2022} discusses skew duality results in the classical symplectic case. In addition, \cite{comb_duality_scrimshaw} studies skew Howe duality for various classical reductive dual pairs.

We note an important application of our results. With the skew quantum Howe duality results of this \chorpaper\ in hand, we may describe generators of a braid group representation on $\End_{\Uqdn}\left(S^{\otimes m})\right)$ arising from the braiding on the category of finite-dimensional modules over $\Uqdn$ explicitly in terms of quantum Clifford algebra operators. In other words, we use our skew quantum Howe duality results to construct solutions of the Yang-Baxter equation that centralize the $\Uqdn$-action on $\bigwedge_q(V^{(nm)})$. These solutions and associated braid group representations have been considered by Rowell and Wang, and by Rowell and Wenzl, in \cite{rowell_wang_2011} and \cite{rowell_wenzl_2017}. They describe the quantum compututation model based on \textit{metaplectic} anyons \cite{hnw_13,hnw_14}. We hope the explicit description made possible by our duality \Cref{uqodn uqprime duality} paves the way for a detailed study of the associated braid representations that can be used to prove the conjectures in \cite{cui_wang_2015} regarding the universality of the quantum computation model based on metaplectic anyons. 

To sum up, this \chorpaper\ develops \textit{operator commutant versions} of the First Fundamental Theorem of invariant theory, in the spirit of \cite[Section~4.3.4]{howe1995}, for the quantum group $\Uqson$ and its spin module. The \Cref{classical_son_case} recalls relevant details of the classical case and \Cref{q orthog duality} develops new results in the quantum setting in three steps. First, \Cref{spin_action_braided_ext_alg} explains how the spin action factors through a quantized Clifford algebra. Then \Cref{commuting_q_embeddings_bd} obtains an action of a coideal subalgebra $\Uqprime$ in $\Uqglm$ that commutes with the $\Uqson$-action on tensor powers $S^{\otimes m}$ of its spin module. Finally, \Cref{q_mf_decomposition_bd} achieves a multiplicity-free decomposition of $S^{\otimes m}$ by constructing joint highest weights with respect to the action of $\Uqson \otimes \Uqprime$.

\subsection*{Acknowledgements} 

This article emerged as part of the author's dissertation work under the supervision of Dan Bump. The author would like to thank Dan Bump for his infinite patience and support, and for the continuous stream of advice that made this work possible. The author would also like to thank Eric Rowell for first pointing out the problem in question. This article benefits from joint work with Travis Scrimshaw on quantum Clifford algebras, which is in preparation, and from his relevant code available on {\sc SageMath} \cite{sagemath}. 

\section{Notation and conventions}
\label{not and conv}
In this article we use the notation fixed in \cite[Section~\labelcref{not and conv}]{willie_a}. For convenience and concreteness, we recall that if $\lieg$ is a complex semisimple Lie algebra there exists a unique non-degenerate symmetric invariant bilinear form $\langle , \rangle\colon \lieg \times \lieg \to \mathbb{C}$ such that
\begin{gather*}
\begin{split}
	\langle H_i, H_j \rangle = \inv{d_j} a_{ij}, \quad \langle H_i, E_j \rangle = \langle H_i, F_j \rangle = 0, \\
	\langle E_i, E_j \rangle = \langle F_i, F_j \rangle =  0,  \quad \text{and} \quad \langle E_i, F_j \rangle = \inv{d_i} \delta_{ij},
\end{split}
\end{gather*}
for all $i, j$ \cite[Theorem~2.2]{kac_1985}. When $\lieg = \gln$ we take $\langle , \rangle$ to be the non-degenerate trace bilinear form of the natural representation. The form is normalized so that
\begin{equation}\label{form normalization}
	\langle \alpha, \alpha \rangle = 2
\end{equation}
for \textit{short} roots. In the same spirit we record some relevant Cartan matrices: 
\begin{align}\label{cartan mats}
\begin{gathered}
	A_n = 
	\begin{bmatrix*}[r]
	2  & -1     &        &        & \\
	-1 & 2      & -1     &        & \\
	   & \ddots & \ddots & \ddots & \\
	   &        & -1     & 2      & -1 \\
	   &        &        & -1     & 2
	\end{bmatrix*}, 
	\\ \\
	B_n = 
	\begin{bmatrix*}[r]
	\begin{matrix}
	& & &&\\
	& A_{n-1} &&& \\
	& & &&
	\end{matrix} \rvline
	& 
	\begin{matrix*}[c]
	0 \\ \vdots \\ 0 \\ -1
	\end{matrix*} \\
	\cmidrule(lr){1-1}
	\begin{matrix*}[l]
	0 & \cdots & 0 & -2
	\end{matrix*}
	& 2
	\end{bmatrix*}, 
	\quad \text{and} \quad 
	D_n = 
	\begin{bmatrix*}[r]
	\begin{matrix}
	& & &&&\\
	& A_{n-1} &&&& \\
	& & &&&
	\end{matrix} \rvline
	& 
	\begin{matrix*}[c]
	0 \\ \vdots \\ 0 \\ -1 \\ 0
	\end{matrix*} \\
	\cmidrule(lr){1-1}
	\begin{matrix*}[l]
	0 & \cdots & 0 & -1 & 0
	\end{matrix*}
	& 2
	\end{bmatrix*}.
\end{gathered}
\end{align}
We label the matrices by the Lie type of the root system to which they are associated. The corresponding diagonal   root lengths matrices are described by $d = (1, \ldots, 1)$ for the root systems of types $A_n, D_n$, and by $d = (2, \ldots, 2, 1)$ for type $B_n$.

\section{The classical case}
\label{classical_son_case}
In this section we develop orthogonal duality theory in the classical case. In particular, we (re)-prove the classical $O(n) \times SO(m)$-duality \Cref{on som mf decomp}, this time using a double centralizer property inside a Clifford algebra. This theorem is well-known to experts but our method lays the foundations for our treatment of the more difficult quantum case in \Cref{q orthog duality}. 

We prove \Cref{on som mf decomp} in three steps. First, in \Cref{orthog_actions_ext_alg} we show that for any complex vector space $V$ there are actions of $\mathfrak{so}(V \oplus V^*)$ and of $\mathfrak{so}(V)$ on the exterior algebra $\bigwedge(V)$ that factor through the Clifford algebra $Cl(V \oplus V^*)$. Then we construct commuting embeddings of $\mathfrak{so}_{2n}$ and $\mathfrak{so}_m$ into $\Clnm$ in \Cref{comm orthog actions on ext alg}. Finally, we compute a multiplicity-free decomposition of $\bigwedge(\mathbb{C}^{nm})$ as an $\mathfrak{o}_{2n} \otimes \mathfrak{so}_m$-module in \Cref{mf_decomposition_bd}.

As in \crossrefglnch, we work at the Lie algebra level throughout, since we are most interested in the quantum case and it is the enveloping algebra $\Ug$, rather than the Lie group $G = \exp(\lieg)$, that more closely resembles $\Uqg$.

	\subsection{Two orthogonal actions on the spin module $S$}
	\label{orthog_actions_ext_alg}
	Consider any complex vector space $V$ and let $(\cdot, \cdot)$ denote the symmetric bilinear form on $V \oplus V^*$ arising from the dual pairing between $V$ and $V^*$; explicitly, 
\begin{equation*}
	\beta\big((v, f), (w, h)\big) = f(w) + h(v), \quad v, w \in V, \,\, f, h \in V^*.
\end{equation*}
In \crossrefcliff{cln props} we review the Clifford algebra $Cl(V \oplus V^*)$ on $V \oplus V^*$ and its spin action on the exterior algebra $\bigwedge(V)$ via inner and exterior multiplication operators $\iota_f$ and $\varepsilon_v$. This treatment is fairly standard and may be found in various sources, e.g. \cite[Chapter~31]{Bump} or \cite[Chapter~6]{GW}. When necessary or convenient, we view $Cl(V \oplus V^*)$ as a Lie algebra with bracket given by the usual algebra commutator: $[A, B] = AB - BA$.

There are two actions of orthogonal Lie algebras on $\bigwedge(V)$. On one hand, there is a \textit{spin} action of $\mathfrak{so}(V \oplus V^*)$ on $\bigwedge(V)$ that makes the following diagram commute. On the other, there is an $\mathfrak{so}(V)$-action obtained by restricting the $\mathfrak{gl}(V)$-action on $\bigwedge(V)$, described in \crossrefgln{ext_alg_as_cln_module}, to a subalgebra of skew-symmetric operators.
\begin{equation}\label{spin action compatible with glv}
\begin{tikzcd}
	\mathfrak{so}(V \oplus V^*) \rar
	& \End\left(\bigwedge(V)\right) \\
	\mathfrak{gl}(V) \uar \ar[ur]
\end{tikzcd}
\end{equation}

In this \chorpaper\ we are interested in these actions in so far as they motivate results in the quantized setting. Therefore since $U_q(\mathfrak{so}_m)$ does \textit{not} embed into $\Uqglm$, we would not gain much by studying the action on $\bigwedge(V)$ of a positive Borel subalgebra in $\mathfrak{so}(V)$ with respect to a chosen weight basis.

Thus in this subsection we focus on the spin action of $\mathfrak{so}(V \oplus V^*)$, which we quantize in \Cref{spin_action_braided_ext_alg} using a map $\Uqdn \to Cl_q(n)$. This action factors through $Cl(V \oplus V^*)$ as follows. Let $\{A, B\} = AB + BA$ and consider the following simple identity, valid in any associative algebra containing $V \oplus V^*$:
\begin{equation}\label{adjoint action by anticommutators}
	[[v, w], u] = v \{w, u\} + \{w, u\} v - w \{v, u\} - \{v, u\} w 
	\qquad u, v, w \in V \oplus V^*.
\end{equation}
As an identity in $Cl(V \oplus V^*)$, \Cref{adjoint action by anticommutators} reads
$$[[v, w], u] = 2 \big(( w, u ) v - ( v, u ) w\big).$$
A straightforward calculation shows the operator $X_{v, w}(u) = ( w, u ) v - ( v, u ) w$ is skew-symmetric with respect to the form $(\cdot, \cdot)$, so we obtain an isomorphism of Lie algebras $\mathfrak{so}(V \oplus V^*) \to Cl(V \oplus V^*)$.

\begin{lem}{\cite[Lemma~6.2.2]{GW}}\label[lem]{spin so2n embedding into cl2n}
	Let $\gamma\colon V \oplus V^* \hookrightarrow Cl(V \oplus V^*)$ denote the natural inclusion map\longer{ defined in \Cref{embedding base space into clifford alg}}. There is an injective homomorphism of Lie algebras $\varphi\colon \mathfrak{so}(V \oplus V^*) \to Cl(V \oplus V^*)$ satisfying 
	$$
		\varphi(X_{v, w}) = \frac{1}{2} [\gamma(v), \gamma(w)]
		\quad \text{and} \quad
		[\varphi(X_{v,w}), u] = X_{v,w}(u)
	$$
	for every $u, v, w \in V \oplus V^* \subset Cl(V \oplus V^*)$.
\end{lem}

From now on we let $S = \bigwedge(V)$ denote the \textit{spin module} over $\mathfrak{so}(V \oplus V^*)$ equipped with the action defined by \Cref{spin so2n embedding into cl2n}.

The subspaces $V$ and $V^*$ are Lagrangian in $V \oplus V^*$, so $\mathfrak{so}(V \oplus V^*)$ has a three-step grading of the form
$$\mathfrak{so}(V \oplus V^*) \cong \mathfrak{so}^{(2,0)} \oplus \mathfrak{so}^{(1,1)} \oplus \mathfrak{so}^{(0,2)}.$$
Here
\begin{align}\label{so2n grading}
\begin{split}
	\mathfrak{so}^{(2,0)} &= \vecspan{ \varepsilon_v \varepsilon_w \mid v, w \in V }, \\
	\mathfrak{so}^{(1,1)} &= \vecspan{{1 \over 2}[\varepsilon_v, \iota_f] \mid v \in V, f \in V^*}, \quad \text{and} \\
	\mathfrak{so}^{(0,2)} &= \vecspan{\iota_f \iota_g \mid f, g \in V^*}.
\end{split}
\end{align}
The subspaces $\mathfrak{so}^{(2,0)}$ and $\mathfrak{so}^{(0,2)}$ generate abelian subalgebras, both normalized by $\mathfrak{so}^{(1,1)}$. If we choose an $\mathfrak{so}(V \oplus V^*)$-weight basis of the $2n$-dimensional $V \oplus V^*$ that is isotropic with respect to $(\cdot, \cdot)$ then $\mathfrak{so}^{(1, 1)}$ corresponds to the set of $(2n) \times (2n)$ block diagonal matrices with blocks of size $n \times n$.

\begin{rmk}
	The notation $\mathfrak{so}^{(0, 2)}$, $\mathfrak{so}^{(1,1)}$, and $\mathfrak{so}^{(2,0)}$ is motivated by the following observations. Each element in $\mathfrak{so}^{(0,2)}$ is a product of two \textit{lowering operators} mapping $\bigwedge^p (W)$ into $\bigwedge^{p-2}(W)$. Similarly, each element of $\mathfrak{so}^{(2,0)}$ is a product of two \textit{raising operators} taking $\bigwedge^p (W)$ into $\bigwedge^{p+2} (W)$. The elements of $\mathfrak{so}^{(1,1)}$ are products of a raising and a lowering operator, thus preserving components of homogeneous degree.
\end{rmk}

The subalgebra $\mathfrak{so}^{(1, 1)}$ is isomorphic to $\mathfrak{gl}(V)$. Recall that $\mathfrak{gl}(V) \cong V \otimes V^*$ and $\{\varepsilon_v, \iota_f \} = f(v)$ by \crossrefcliff{inner exterior mult comm}. Therefore
\begin{align}\label{normalized gln action}
	\varphi(X_{v, f}) = {1 \over 2}[\varepsilon_v, \iota_f] = \varepsilon_v \iota_f - \frac{1}{2}f(v),
\end{align}
showing $\mathfrak{so}^{(1,1)} \cong \mathfrak{gl}(V)$. In particular, the element $X = v \otimes f$ in $V \otimes V^* \cong \mathfrak{gl}(V) \subset \mathfrak{so}(V \oplus V^*)$ acts on $\bigwedge(V)$ by the rightmost operator in \Cref{normalized gln action}.

Note that the $\mathfrak{so}^{(1,1)} \cong \mathfrak{gl}(V)$-action on $\bigwedge(V)$ differs from the $\mathfrak{gl}(V)$-action defined by \crossrefgln{glv action by inner and exterior mult} by constants. In particular, $f(v)$ is the trace of the endomorphism $X = v \otimes f \in \mathfrak{gl}(V)$ because $X$ is rank 1 and its only non-trivial eigenvalue is $f(v)$. These constants do not alter commutation relations, but they do affect the action of both $\mathfrak{gl}(V)$ and $GL(V)$ on $\bigwedge(V)$. At the Lie group level, subtracting half the trace normalizes the $GL(V)$-action by a factor of $\det^{-1/2}$. One consequence of this normalization is that the action of $GL(V)$, or rather its two-fold cover, on $\bigwedge(V)$, is now self-contragradient \cite{howe1995}.

\Cref{spin so2n embedding into cl2n} describes an $\mathfrak{so}(V \oplus V^*)$-action on $\bigwedge(V)$ that factors through $Cl(V \oplus V^*)$. We are interested in a quantum version of this action. In the quantum setting operators are described by their action on a weight basis, so we describe the map defined by \Cref{spin so2n embedding into cl2n} explicitly in terms of an $\mathfrak{so}(V \oplus V^*)$-weight basis in preparation of the quantum case. 

To begin, let $v_1, \ldots, v_n$ denote a basis of $V$ and let $v_{-n}, \ldots, v_{-1}$ denote the corresponding dual basis of $V^*$, chosen so that $v_{-j}(v_i) = \delta_{ij}$. Then $v_1, \ldots, v_n, v_{-n}, \ldots, v_{-1}$ is an isotropic basis of $V \oplus V^*$ and there is an injective map $\mathfrak{so}(V \oplus V^*) \to \mathrm{Mat}_{2n}(\mathbb{C})$ satisfying
\begin{align}\label{so2n mat repn}
\begin{split}
	E_i &\to M_{i, i+1} - M_{-i-1, -i}, 
	\qquad \,\,\,\,\; 
	F_i \to M_{i+1, i} - M_{-i, -i-1}, \quad i < n \\
	E_n &\to M_{n-1, -n} - M_{n, -n+1}, 
	\qquad 
	F_n \to M_{-n, n-1} - M_{-n+1, n} \\
	H_i &\to M_{ii} + M_{-i-1, -i-1} - \left(M_{i+1,i+1} - M_{-i,-i}\right), \quad i < n \\
	H_n &\to M_{n-1,n-1} + M_{nn} - \left(M_{-n+1,-n+1} - M_{-n,-n}\right).
\end{split}
\end{align}
The $M_{ij}$ denote matrix units with respect to the $v_i$ basis defined by $M_{ij} v_k = \delta_{jk} v_i$. The action of the $H_i$ is diagonal, so $v_1, \ldots, v_n, v_{-n}, \ldots, v_{-1}$ is in fact simultaneously a $\mathfrak{gl}(V)$- and an $\mathfrak{so}(V \oplus V^*)$-weight basis. Our choice of weight basis defines an isomorphism $V \oplus V^* \cong \mathbb{C}^n \oplus (\mathbb{C}^n)^*$ and from now we denote $\mathfrak{so}(V \oplus V^*)$ by $\mathfrak{so}_{2n}$. Under the isomorphism defined by \Cref{so2n mat repn}, the image of $\mathfrak{so}_{2n}$ is the set of traceless matrices $X$ satisfying $X J + J X^T = 0$, with $J \coloneqq 
\begin{bmatrix}
 & & 1 \\
 & \iddots & \\
 1 & & 
\end{bmatrix}$.

We obtain a basis for $\bigwedge(V)$ using the vectors $\bar{v}(\ell)$ defined by \crossrefgln{bar v ext alg basis review}. On the Clifford algebra side, we consider the generators $\psi_i = \iota_{v_{-i}}$ and $\psi_i^\dagger = \varepsilon_{v_i}$, for $i = 1, \ldots, n$, much like in \crossrefgln{ext_alg_as_cln_module}.  For convenience, we recall that the $\psi_i$ and $\psi_j^\dagger$ satisfy the canonical anticommutation relations\longer{ \eqref{cac}}
\begin{align*}
	\begin{gathered}
		\psi_i \psi_j + \psi_j \psi_i = \psi_i^\dagger \psi_j^\dagger + \psi_j^\dagger \psi_i^\dagger = 0
		\quad \text{and} \\
		\psi_i^{\pdg} \psi_j^\dagger + \psi_j^\dagger \psi_i^{\pdg} = \delta_{ij},
	\end{gathered}
\end{align*}
and that they act by lowering and raising operators: for any $\bar{v}(\ell)$ in $\bigwedge(V)$,
\begin{align*}
	\begin{split}
		\psi_i \, \bar{v}(\ell) &= (-1)^{\ell_1 + \cdots + \ell_{i-1}} \bar{v}(\ell - e_i), \text{ and} \\
		\psi_i^\dagger \, \bar{v}(\ell) &= (-1)^{\ell_1 + \cdots + \ell_{i-1}} \bar{v}(\ell + e_i).
	\end{split}
\end{align*}
 
The following proposition defines the map $\mathfrak{so}_{2n} \to \Cln$ explicitly with respect to the $v_i$ basis. 

\begin{prop}\label[prop]{so2n into cl2n embedding}
	Recall the map $\Phi_n \colon \gln \to \Cln$ defined in \crossrefgln{gln cl embedding}. There is a Lie algebra homomorphism  $\Phi_n^{\mathbf{D}} \colon \mathfrak{so}_{2n} \to \Cln$ satisfying $\Phi_n^{\mathbf{D}}(X) = \Phi_n(X)$ whenever $X \in \gln \subset \mathfrak{so}_{2n}$ and
	\begin{align*}
		E_n &\to \psi_{n-1}^\dagger \psi_n^\dagger, \\
		F_n &\to \psi_{-n} \psi_{-n+1}, \\
		H_n &\to \psi_{n-1}^\dagger \psi_{-n+1}^{\pdg} + \psi_n^\dagger \psi_{-n}^{\pdg} - 1.
	\end{align*}
\end{prop}
\begin{proof}
	This is simply the map defined by \Cref{spin so2n embedding into cl2n} in terms of the $\Cln$ generators described by \Cref{cac review,inner ext operator action review}.
\end{proof}

We note that $S = \bigwedge(V)$ is \textit{not} an irreducible $\mathfrak{so}_{2n}$-module. Rather, it is the sum of two irreducible components: as an $\mathfrak{so}_{2n}$-module,
\begin{align*}
	\bigwedge(V) \cong S_+ \oplus S_-.
\end{align*}
Here $S_\pm$ denotes the irreducible $\mathfrak{so}_{2n}$ module with highest weight $\left({1 \over 2}, \ldots, \pm {1 \over 2}\right)$. 

Regardless, $S$ is an irreducible module of the \textit{full orthogonal Lie algebra} $U(\on)$.

\begin{dfn}\label[defn]{on defn}
	The Hopf algebra $U(\mathfrak{o}_n) = U(\mathfrak{so}_n) \rtimes \mathbb{Z}_2$ is generated by the enveloping algebra $U(\mathfrak{so}_n)$ and the additional generator $t$, subject to the following relations. If $n$ is odd, then $t$ commutes with every generator. If $n = 2r$ is even, then
	\begin{align}\label{t gen comm}
	\begin{split}
			t E_{r-1} \inv{t} &= E_r, \quad t E_{r} \inv{t} = E_{r-1}, \\
		t F_{r-1} \inv{t} &= F_r, \quad t F_{r} \inv{t} = F_{r-1}, \\
		t H_{r-1} \inv{t} &= H_r, \quad t H_{r} \inv{t} = H_{r-1},
	\end{split}
	\end{align}
	and $t$ commutes with all other generators. The element $t$ is \textit{group-like}, which means we extend the comultiplication $\Delta$ and the antipode $S$ of $U(\mathfrak{so}_n)$ to $U(\mathfrak{o}_n)$ by specifying that $\Delta(t) = t \otimes t$ and $S(t) = \inv{t}$. 
\end{dfn}

Note that when $n = 2r$ is even, the conjugation action of $t$ on the $\mathfrak{so}_n$ simple positive root vectors induces the $D_r$ Dynkin diagram automorphism swapping the two leaf nodes attached at the trivalent vertex.

In general, we may extend any $\son$-module to a $U(\mathfrak{o}_n)$-module by specifying the action of $t$. The defining relations \eqref{t gen comm}, together with $t^2 = 1$, imply that when $n = 2r + 1$ is \textit{odd}, $t$ must act by $\pm 1$ on each $\son$-weight space. In fact since $t$ commutes with every $\son$ generator, Schur's lemma implies that $t$ must act by a scalar on any irreducible $\son$-module. 

Conversely, when $n = 2r$ is \textit{even}, $t$ induces the map $\mu \to \bar{\mu}$ on the weight lattice of $V$, with $\mu = (\mu_1, \ldots, \mu_r)$ and $\bar{\mu} = (\mu_1, \ldots, \mu_{r-1}, -\mu_r)$: if $v_\mu$ is a weight vector of weight $\mu$, then $t v_\mu$ is a weight vector of weight $\nu$ determined by the relations 
\begin{align*}
	\langle \nu, \alpha_i \rangle &= \langle \mu, \alpha_i \rangle 
	\quad \text{for } i < r-1, \\
	\langle \nu, \alpha_{r-1} \rangle &= \langle \mu, \alpha_{r} \rangle, \\
	\langle \nu, \alpha_{r} \rangle &= \langle \mu, \alpha_{r-1} \rangle,
\end{align*}
which imply $\nu = \bar{\mu}$. Note that if $v_\mu$ is a highest weight vector of the $\son$ action, then $t v_\mu$ is also a highest weight vector with respect to $\son$, since it is annihilated by every $E_i$ generator. Therefore, if $\mu \neq \bar{\mu}$, or equivalently if $\mu_r \neq 0$, any irreducible $\on$-module containing an $\son$-highest weight vector of weight $\mu$ splits into two irreducible $\son$-modules upon restriction. Alternatively, if $\mu = \bar{\mu}$, or equivalently if $\mu_r = 0$, then $t$ preserves the highest weight space and there is an irreducible module for each possible action of $t$. Since $t^2 = 1$, there are exactly two inequivalent irreducible $\on$-modules in this case.

To summarize, the $\on$ and $\son$ representations are related as follows. The irreducible $\son$-modules are parametrized by dominant highest weights $\mu$ satisfying
\begin{align*}
	\mu_1 \geq \mu_2 \geq \cdots \geq \mu_{r-1} \geq |\mu_r|.
\end{align*}
Given a partition $\mu$, let $\mu'$ denote its conjugate and suppose $V_\mu$ is an irreducible $\son$-module with highest weight $\mu$. There are two possibilities.
\begin{enumerate}[(i)]
	\item If $\mu_r = 0$, then there are exactly two non-isomorphic $\on$-modules whose restriction to $\son$ is isomorphic to $V_\mu$: one is labeled by $\mu$ and the other by $\mu^\dagger$. The Young diagram corresponding to $\mu^\dagger$ is identical to the one corresponding to $\mu$ except for its first column, which has $n - \mu_1'$ boxes.
	\item Alternatively, if $\mu_r \neq 0$, then there is exactly one $\on$-module corresponding to $\mu$, and its restriction to $\son$ decomposes as $V_\mu \oplus V_{\bar{\mu}}$. The corresponding $\on$-module is parametrized by the partition $\mu$ with $\mu_r > 0$.
\end{enumerate}

In any case, we see that when $n = 2r$ is even, the irreducible $U(\on)$-modules are parametrized by partitions with at most $n$ parts satisfying
\begin{align}\label{on weight condition}
	\mu_1' + \mu_2' \leq n.
\end{align}
	
	\subsection{Commuting embeddings into the Clifford algebra}
	\label{comm orthog actions on ext alg}
	As in \crossrefgln{commuting_a_embeddings}, now suppose $V = U \otimes W$ with $\dim U = n$ and $\dim W = m$. In this subsection, we construct commuting embeddings of $\mathfrak{so}(U \otimes U^*)$ and $\mathfrak{so}(W)$, and of $\mathfrak{so}(U)$ and $\mathfrak{so}(W \oplus W^*)$, into the Clifford algebra
$$Cl\left((U \otimes W) \oplus (U \otimes W)^*\right) \cong \End\left(\bigwedge(U \otimes W)\right)$$
as in \Cref{classical seesaw}. These embeddings rely on the maps $\lambda\colon \gln \to \Clnm$ and $\rho\colon \glm \to \Clnm$ defined in \crossrefgln{gln embedding into clnm} and \crossrefgln{glm embedding into clnm}. In fact, the constructions in this section are analogous to those in \crossrefgln{commuting_a_embeddings}.

Much like in \crossrefgln{gln ei as matrix units}, we map $\mathfrak{so}(U\oplus U^*)$ into $\mathfrak{so}(U \oplus U^*) \otimes \mathfrak{gl}(W)$, which can be seen as a subalgebra of $\mathfrak{so}(U \oplus U^*) \otimes \mathfrak{so}(W \oplus W^*) \subseteq \mathfrak{so}\left(V \oplus V^*\right)$, by tensoring with the identity. Then we use the map of \Cref{so2n into cl2n embedding} to embed $\mathfrak{so}(V \oplus V^*)$ into $Cl(V \oplus V^*)$. The resulting $\mathfrak{so}(U\oplus U^*)$-action on $S^{\otimes m}$ coincides with the action obtained by composing the spin action described in \Cref{orthog_actions_ext_alg} with the comultiplication. We obtain a commuting action of $\mathfrak{so}(W)$ by restricting the $\mathfrak{gl}(W)$-action on $\bigwedge(V)$ described in \crossrefgln{glm embedding into clnm}.

Dually, we tensor with the identity to embed $\mathfrak{so}(W \oplus W^*)$ into $\mathfrak{gl}(U) \otimes \mathfrak{so}(W \oplus W^*)$, which is a subalgebra of $\mathfrak{so}(U \oplus U^*) \otimes \mathfrak{so}(W \oplus W^*) \subseteq \mathfrak{so}\left(V \oplus V^*\right)$, and then we compose with the map defined in \Cref{so2n into cl2n embedding}. Again, we obtain a commuting copy of $\mathfrak{so}(U) \subset \mathfrak{gl}(U)$ by restricting the action defined in \crossrefgln{gln embedding into clnm}. Alternatively, we may obtain these embeddings by reversing the roles of $U$ and $W$. In this case, the preferred factorization of $\bigwedge(\mathbb{C}^{nm})$ is into $n$ tensor factors of the $\mathfrak{so}(W \oplus W^*)$-spin module $\bigwedge(\mathbb{C}^m)$, instead of $m$ factors of the $\mathfrak{so}(U \oplus U^*)$-spin module $\bigwedge(\mathbb{C}^n)$.

Ultimately we are interested in quantum versions of these embeddings, so we define them explicitly with respect to the $\mathfrak{gl}(V)$-weight basis of $V = U \otimes W$ defined in \crossrefgln{commuting_a_embeddings}. This basis may be extended to the $\mathfrak{so}(V \oplus V^*)$-weight basis $v_1, \ldots, v_n, v_1^*, \ldots, v_n^*$ of $V \oplus V^*$ by appending the corresponding dual basis of $V^*$.

The next proposition describes an explicit embedding $\mathcal{L}$ that takes $\mathfrak{so}_{2n} \coloneqq \mathfrak{so}(U \oplus U^*)$ into $\Clnm$ and makes the following diagram commute. This proposition directly motivates \Cref{uqdn clqnm embedding} in the quantum case.
\begin{equation}\label[diag]{son action diag}
\begin{tikzcd}[column sep=1.45cm, row sep=1.1cm]
	\mathfrak{sl}_n \dar \ar[r, "\Delta^{(m-1)}"] 
	& \mathfrak{sl}_n^{\otimes m} \dar \ar[r, "\Phi_n^{\otimes m}"]
	& \Cln^{\otimes m} \rar \ar[d, "\id"]
	& \End\left(\bigwedge(\mathbb{C}^n)^{\otimes m} \right) \ar[d, "\id"] \\
	\mathfrak{so}_{2n} \ar[r, "\Delta^{(m-1)}"] \ar[drr, dashed, blue!60, "\mathcal{L}"]
	& \mathfrak{so}_{2n}^{\otimes m} \ar[r, "\left(\Phi_n^{\mathbf{D}}\right)^{\otimes m}"]
	& \Cln^{\otimes m} \rar \ar[d, "\Gamma_n"] 
	& \End\left(S^{\otimes m}\right) \dar \\
	\mathfrak{sl}_n \ar[rr, swap, "\lambda"] \uar
	& 
	& \Clnm \rar
	& \End\left(\bigwedge(\mathbb{C}^{nm})\right) 
\end{tikzcd}
\end{equation} 
Here $\Delta \colon \Ug \to \Ug^{\otimes 2}$ denotes the comultiplication in the enveloping algebra. Recall \Cref{so2n into cl2n embedding} defines $\Phi_n^{\mathbf{D}}$. We take $\Gamma_n$ as in \crossrefcliff{cl tensor m into clnm} and $\lambda$ and $\Phi_n$ as in \crossrefgln{gln cl embedding,gln embedding into clnm}.

\begin{prop}\label[prop]{so2n embedding into clnm}
	There is a Lie algebra homomorphism $\mathcal{L}\colon\mathfrak{so}_{2n} \to \Clnm$ satisfying $\mathcal{L}(X) = \lambda(X)$ for every $X$ belonging to the subalgebra $\gln \subset \mathfrak{so}_{2n}$ and  
	\begin{align*}
		E_n &\to \sum_{j = 1}^m \psi_{n-1 + (j-1)n}^\dagger \psi_{n + (j-1)n}^\dagger,\\
		F_n &\to \sum_{j = 1}^m \psi_{n + (j-1)n} \psi_{n-1 + (j-1)n}, 
		\quad \text{and} \\
		H_n &\to -m + \sum_{j=1}^m \left( 
			\psi_{n-1 + (j-1)n}^\dagger \psi_{n-1 + (j-1)n} 
		+	\psi_{n + (j-1)n}^\dagger \psi_{n + (j-1)n} 
		\right).
	\end{align*}
\end{prop}
\begin{proof}
	This map is the composition $\mathcal{L} = \Gamma_n \circ (\Phi_n^{\mathbf{D}})^{\otimes m} \circ \Delta^{(m-1)}$ of known Lie algebra maps illustrated in \Cref{son action diag}.
\end{proof}

An immediate corollary of \Cref{so2n embedding into clnm} is that $S^{\otimes m} \cong \bigwedge(\mathbb{C}^{nm})$ as an $\mathfrak{so}_{2n}$-module.

As explained in \Cref{spin_action_braided_ext_alg}, we will not gain much in the quantum case by describing the commuting embedding of $\mathfrak{so}(W)$ into $Cl(V \oplus V^*)$ explicitly in terms of root vectors because the analogous commuting factor in the quantum case is the \textit{non-standard deformation} $\Uqprime$, which does not have an analogue of a positive Borel subalgebra. Thus we avoid discussing the embedding of $\mathfrak{so}(W)$ here.
	
	\subsection{Multiplicity-free decomposition of $S^{\otimes m}$}
	\label{mf_decomposition_bd}
	In this section we compute a multiplicity-free decomposition of $\bigwedge(\mathbb{C}^{nm})$ as a $U(\on) \otimes U(\mathfrak{so}_m)$-module. Recall that $S \cong \bigwedge(\mathbb{C}^n)$ and $\bigwedge(\mathbb{C}^{nm}) \cong \bigwedge(\mathbb{C}^n)^{\otimes m}$ as a $U(\on)$-module.
	
\Cref{orthog_actions_ext_alg} explains that the irreducible representations of $U(\on)$ are parametrized by partitions $\mu$ such that $\mu_1' + \mu_2' \leq n$, with $\mu'$ denoting the conjugate of $\mu$. When $m = 2r$ is \textit{even}, the irreducible representations of $\mathfrak{so}_m$ are labeled by dominant  weights $\nu$ such that $\nu_1 \geq \nu_2 \geq \cdots \geq |\nu_r|$. Conversely, if $m = 2r+1$ is odd, the irreducible representations of $\mathfrak{so}_m$ are labeled by dominant weights $\nu$ such that $\nu_1 \geq \nu_2 \geq \cdots \geq \nu_r \geq 0$. In any case, either all $\nu_j$ are integers or all $\nu_j \equiv 1/2 \mod \mathbb{Z}$. 

The map 
\begin{align}\label{mu bar}
	\mu \to \bar{\mu}, \quad \text{with} \quad \bar{\mu}_i = {n \over 2} - \mu_{r + 1 - i}'
\end{align}
defines a bijection between the set of irreducible representations $V_\mu$ of $U(\on)$ for which $\mu_1 \leq r$ and the set of irreducible $\mathfrak{so}_m$-representations $V_{\bar{\mu}}$ for which $\bar{\mu}_1 \leq n/2$ and $n/2 - \bar{\mu}_i$ is an integer for $1 \leq i \leq r$.

\begin{thm}\label{on som mf decomp}
	Let $S \cong S_+ \oplus S_-$ denote the $\mathfrak{so}_{2n}$-spin module. As a $U(\mathfrak{o}_{2n}) \otimes U(\mathfrak{so}_m)$-module, $S^{\otimes m} \cong \bigwedge(\mathbb{C}^{nm})$ is multiplicity-free. In particular, we have
	$$S^{\otimes m} \cong \bigoplus_\mu V_\mu^{(n)} \otimes V_{\bar{\mu}}^{(m)}$$
	as a $U(\mathfrak{o}_{2n}) \otimes U(\mathfrak{so}_m)$-module. The sum ranges over all partitions $\mu$ that fit in a $(2n) \times r$ rectangle and satisfy $\mu_1' + \mu_2' \leq 2n$. In the decomposition $V_\mu^{(n)}$ denotes the irreducible $U(\mathfrak{o}_{2n})$-module with highest weight and $\mu$, $V_\nu^{(m)}$ denotes an irreducible $\mathfrak{so}_m$-module indexed by $\nu$. Consequently, $U(\mathfrak{o}_{2n})$ and $U(\mathfrak{so}_m)$ generate mutual commutants in $\End\left(\bigwedge(\mathbb{C}^{nm})\right)$.
\end{thm}
\begin{proof}
	In \crossrefgln{skew gln glm duality} we compute the decomposition of $\bigwedge(\mathbb{C}^{nm})$ into isotypic components of $\gln \otimes \glm$. Notice the subspace $\mathfrak{so}^{(1, 1)} \subset \mathfrak{so}_{2n}$ described in \Cref{so2n grading} is a normalization of $\gln$. In addition, it is the Levi component of the parabolic subalgebra $\mathfrak{so}^{(1, 1)} \oplus \mathfrak{so}^{(0, 2)}$ of $\mathfrak{so}_{2n}$, whose nilradical is $\mathfrak{so}^{(0, 2)}$. Thus in any irreducible $\mathfrak{so}_{2n}$-module $M$ the space 
	$$
		\ker \mathfrak{so}^{(0, 2)} 
		= \{v \in M \mid X v = 0, \text{ for all } X \in \mathfrak{so}^{(0, 2)} \}
	$$
	is an irreducible $\mathfrak{so}^{(1, 1)}$-module. Moreover, the irreducible representation of the Levi component $\mathfrak{so}^{(1, 1)}$ characterizes the $\mathfrak{so}_{2n}$-module containing it. This means the irreducible $\mathfrak{so}_{2n}$-modules appearing in the decomposition of $\bigwedge(\mathbb{C}^{nm})$ are the $\gln$-isotypic components that appear in the skew $\gln \otimes \glm$-duality \crossrefgln{skew gln glm duality} and are annihilated by $\mathfrak{so}^{(0, 2)}$. These modules are parametrized by a subset of dominant weights fitting in an $n \times m$ rectangle. 
	
	The condition $\mu_1' + \mu_2' \leq 2n$ arises when we consider the relationship between $\on$ and $\son$-modules, explained in \Cref{orthog_actions_ext_alg}. The restriction $\mu_1 \leq r$ arises when we consider the commuting $\mathfrak{so}_m$-action. Corollary ~3.3.2 in \cite{howe1995} proves that an irreducible $\glm$-module parametrized by $\mu$ contains a highest weight with respect to the subalgebra $\mathfrak{so}_m \subset \glm$ only if every row of $\mu$ is even. Thus the irreducible $\glm$-modules appearing in the decomposition of the skew $\gln \otimes \glm$-duality result\longer{ \Cref{skew gln glm duality}} that are isotypic with respect to the restricted $\mathfrak{so}_m$-action correspond to $\mathfrak{so}_{2n}$-modules with highest weight $\mu$ satisfying $\mu_1 \leq r$.
	
	The exact correspondence can be computed using an explicit Borel subalgebra for $\mathfrak{so}_m$. This calculation does not have a quantum analogue, so we omit it here.
\end{proof}

Of course we obtain an analogous decomposition of $\mathfrak{so}_n \otimes \mathfrak{o}_{2m}$-modules by considering the isomorphism $S_m^{\otimes n} \cong \bigwedge(\mathbb{C}^m)^{\otimes n} \cong \bigwedge(\mathbb{C}^{nm})$ of $\mathfrak{so}_{2m}$-modules. Here we let $S_m$ denote the $\mathfrak{so}_{2m}$-spin module.

\section{The quantum case}
\label{q orthog duality}
In this section we prove quantized skew duality results for Types $\mathbf{BD}$ using our constructions for Type $\mathbf{A}$ from \crossrefgln{q skew duality type a}. \Cref{uqodn uqprime duality} is our main result. In particular, we identify $\Uqgln$ as a subalgebra of $\Uqdn$ and we realize $U_q'(\mathfrak{so}_m)$ as a subalgebra of $\Uqglm$ to extend our $\Uqgln \otimes \Uqglm$-duality \crossrefgln{uqgln uqglm duality} to the orthogonal setting via the seesaw depicted in \Cref{quantum seesaw}. Alternatively, one could realize $U_q'(\mathfrak{so}_n)$ as a subalgebra of $\Uqgln$ and extend the $\Uqglm$-action to a $U_q(\mathfrak{so}_{2m})$-action in order to prove a 

In \crossrefgln{q skew duality type a} we learned that the actions of $\gln$ and $\glm$ on $\bigwedge(\mathbb{C}^{nm})$ can be generalized to the quantum setting to obtain actions of $\Uqgln$ and $\Uqglm$ on $\bigwedge_q(V^{nm})$ by understanding the action of generating root vectors as products of Clifford algebra operators. In the quantum case we have an analogue $\Uqgln \to \Uqdn$ of the diagonal embedding $\mathfrak{gl}(V) \to \mathfrak{so}(V \oplus V^*)$, so we emulate the strategy of \crossrefgln{q skew duality type a}: in \Cref{commuting_q_embeddings_bd} we construct a map $\mathcal{L}_q\colon \Uqdn \to Cl_q(nm)$ that restricts to the map $\lambda\colon \Uqgln \to Cl_q(nm)$ defined in \crossrefgln{commuting quantum actions section} on the subalgebra $\Uqgln \subset \Uqdn$. However, $\Uqson$ does \textit{not} embed into $\Uqgln$, so we cannot directly quantize the action of simple root vectors in $\mathfrak{so}(V)$. 

This is a remarkable feature of the quantum case. The non-standard deformation of $\mathfrak{so}_m$ does not support an analogue of a Borel subalgebra. However, there is an analogue of a Cartan subalgebra in $\Uqprime$ and every irreducible $\Uqprime$-module has a basis indexed by Gelfand-Tsetlin patterns as in the classical case. Thus in the quantum case, the main obstacle in finding joint highest weight vectors in order to decompose $S^{\otimes m}$ is diagonalizing the Cartan subalgebra. This is achieved in \Cref{q_mf_decomposition_bd}.

	\subsection{The spin module $S$ as a braided exterior algebra}
	\label{spin_action_braided_ext_alg}
Recall the notation of \crossrefgln{braided_ext_alg}. In particular, let $V^{(p)}$ denote the natural $U_q(\mathfrak{gl}_p)$-module and consider the braided exterior algebra $\bigwedge_q(\Vn)$ defined by \crossrefgln{braided ext alg as tensor alg quotient} using the $R$-matrix of $\Uqgln$.

In this subsection, we define actions of the orthogonal quantum groups $\Uqdn$ and $\Uqbn$ on $\bigwedge_q(\Vn)$. These actions factor through the $Cl_q(n)$-action on $\bigwedge_q(\Vn)$ defined in \crossrefgln{braided_ext_alg} and they are compatible with the $\Uqgln$-module algebra structure defined in \crossrefgln{uqgln clqn embedding}, in the sense that the following diagrams commute:
\begin{equation}
	\begin{tikzcd}[column sep=large]
		\Uqsln \arrow[hookrightarrow]{r} \ar[dr, swap, "\Phi_{q, n}"]
		& \Uqdn \ar[d, "\mathcal{D}_n"] \\
		& Cl_q(n)
	\end{tikzcd}
	\qquad 
	\begin{tikzcd}[column sep=large]
		\Uqsln \arrow[hookrightarrow]{r} \ar[dr, swap, "\Phi_{q, n}"]
		& \Uqbn \ar[d, "\Phi_{q, n}^{\mathbf{B}}"] \\
		& Cl_q(n)
	\end{tikzcd}
\end{equation}
\Cref{uqdn clqn embedding,uqbn clqn embedding} define the maps $\mathcal{D}_n$ and $\Phi_{q, n}^{\mathbf{B}}$. In addition, the $\Uqdn$-action on $\bigwedge_q(\Vn)$ motivates the embedding of $\Uqdn$ into the quantum Clifford algebra $Cl_q(nm)$ presented in \Cref{commuting_q_embeddings_bd}.

Much like the quantum group $\Uqgln$, the orthogonal quantum groups $\Uqdn$ and $\Uqbn$ also map into quantum Clifford algebra $Cl_q(n)$. These embeddings thereby define \textit{spin} actions on $\bigwedge_q(\Vn)$. Note that in this section we still deal with the braided exterior algebra $\bigwedge_q(\Vn)$ as in \crossrefgln{braided ext alg as tensor alg quotient} using the $\Uqgln$, and \textit{not} the $\Uqdn$, $R$-matrix. Regardless, we will no longer consider the underlying algebra structure of $\bigwedge_q(\Vn)$: we merely extend the $\Uqgln$-module structure.

\begin{prop}\label[prop]{uqdn clqn embedding}
	Recall the map $\Phi_{q, n}$ of \crossrefgln{uqgln clqn embedding}. There is an algebra map $\Phi_{q, n}^{\mathbf{D}} \colon \Uqdn \to Cl_q(n)$ satisfying $\Phi_{q, n}^{\mathbf{D}}(X) = \Phi_{q, n}(X)$ for $X$ belonging to the subalgebra $\Uqsln \subset \Uqdn$ and
	\begin{align*}
		\Phi_{q, n}^{\mathbf{D}}(E_n) &= \psi_{n-1}^\dagger \psi_n^\dagger \\
		\Phi_{q, n}^{\mathbf{D}}(F_n) &= \psi_n \psi_{n-1}^{\pdg}  \\
		\Phi_{q, n}^{\mathbf{D}}(K_n) &= \inv{(q\omega_{n-1}\omega_n)}.
	\end{align*}
\end{prop}
\begin{rmk}
	As in the classical case, we use the \textit{normalized} general linear action. In particular, recall \Cref{normalized gln action} and the comments surrounding it. The normalized action differs from the $\gln$-action defined by \crossrefgln{glv action by inner and exterior mult} only on matrices with non-trivial trace. Thus the two in fact coincide on the subalgebra $\mathfrak{sl}_n \subset \mathfrak{so}_{2n}$. In addition, we note there is an algebra map $\Uqdn \to Cl_q(n)$ satisfying
	\begin{align*}
		E_n &\to \varepsilon_{n-1}^q \varepsilon_n^q 
			= \qinv \left(\prod_{p = 1}^{n-2} \omega_p^{-2}\right) \omega_{n-1}^{-1} 
			  \psi_{n-1}^\dagger \psi_n^\dagger \\
		F_n &\to \iota_{n-1}^q \iota_n^q 
			= \left(\prod_{p = 1}^{n-2} \omega_p^2\right) \omega_{n-1} 
			  \psi_n \psi_{n-1} \\
		K_n &\to \inv{(q\omega_{n-1}\omega_n)}.
	\end{align*}
	This map makes $\bigwedge_q(\Vn)$ into a $\Uqdn$-module algebra in the sense of \cite[Definition~4.1.1]{montgomery_1993}.\longer{ For convenience, \Cref{module alg defn} recalls the definition of a $\Uqg$-\textit{module algebra}.} In this section we do not consider the underlying algebra structure of the $\Uqdn$-module $\bigwedge_q(\Vn)$, so we focus on the action defined in \Cref{uqdn clqn embedding}.
\end{rmk}
\begin{proof}
	The claim follows from a calculation. Since $\Phi_{q, n}$ is an algebra map, it suffices to check the relations involving the images $\widetilde{E}_n, \widetilde{F}_n, \widetilde{K}_n$ of $E_n, F_n, K_n$ under $\mathcal{D}_n$. Let $A = [a_{ij}]$ denote the $n \times n$ Cartan matrix of $\mathfrak{so}_{2n}$, as in \Cref{cartan mats}. 
	
	First notice that $\widetilde{K}_i \widetilde{E}_n \inv{\widetilde{K}_i} = \widetilde{E}_n$ and similarly $\widetilde{K}_n \widetilde{E}_i \inv{\widetilde{K}_n} = \widetilde{E}_i$ whenever $i < n - 2$. When $i \geq n-2$, $K_i$ and $E_n$ contain non-commuting factors, and we calculate that 
	\begin{align*}
		\widetilde{K}_i \widetilde{E}_n \inv{\widetilde{K}_i} 
			&= \qinv 
				\left(\omega_{i}^{-1} \omega_{i+1}\right)
				\psi^\dagger_{n-1} \psi_{n}^\dagger
				\left(\omega_{i}^{-1} \omega_{i+1}\right)^{-1}
			= 
			\begin{cases*}
				q^{-1} \widetilde{E}_n, & if $i = n-2$ \\
				\widetilde{E}_n, & if $i = n-1$.
			\end{cases*}
	\end{align*}
	Similarly, 
	\begin{align*}
		\widetilde{K}_n \widetilde{E}_i \inv{\widetilde{K}_n} 
			&= \qinv \inv{\omega_i}
				\left(\omega_{n-1}^{-1} \omega_{n}^{-1}\right)
				\psi^\dagger_{i} \psi_{i+1} 
				\left(\omega_{n-1} \omega_{n}\right)
			= 
			\begin{cases*}
				q^{-1} \widetilde{E}_i, & if $i = n-2$ \\
				\widetilde{E}_i, & if $i = n-1$.
			\end{cases*}
	\end{align*}
	Finally, we calculate that
	\begin{align*}
		\widetilde{K}_n \widetilde{E}_n \inv{\widetilde{K}_n} 
			&= \qinv 
				\left(\omega_{n-1}^{-1} \omega_{n}^{-1}\right)
				\psi^\dagger_{n-1} \psi_{n}^\dagger
				\left(\omega_{n-1} \omega_{n}\right)
			= q^2 \widetilde{E}_n.
	\end{align*}
	Combining results, we conclude that $\widetilde{K}_i \widetilde{E}_j \inv{\widetilde{K}_i} = q^{a_{ij}} \widetilde{E}_j,$
	as desired.	Applying the $*$-operation defined by \crossrefcliff{star struct} shows that $\widetilde{K}_i \widetilde{F}_j \inv{\widetilde{K}_i} = q^{-a_{ij}} \widetilde{F}_j$ as well.
	
	Using \crossrefcliff{psi psi psid psid comm}, we find that
	\begin{align*}
		[\widetilde{E}_n, \widetilde{F}_n] 
			= [\psi^\dagger_{n-1} \psi_{n}^\dagger, \psi_{n}^{\pdg}\psi_{n-1}^{\pdg}]
			= - {(q \omega_{n-1} \omega_n) - \inv{(q \omega_{n-1} \omega_n)} \over q - \qinv} 
			= {\widetilde{K}_n - \inv{\widetilde{K}_n} \over q - \qinv}.
	\end{align*}
	
	To conclude, we verify the quantum Serre relations. If $i \neq n-2$, then $[\widetilde{E}_n, \widetilde{E}_i] = 0$ because $\widetilde{E}_n$ and $\widetilde{E}_i$ share no common $\phi_i$ factors, with $\phi_i$ denoting either $\psi_i$ or $\psi_i^\dagger$ as usual. In addition, $[\widetilde{E}_n, \widetilde{E}_{n-1}] = 0$ because $\widetilde{E}_n$ and $\widetilde{E}_{n-1}$ share a common $\psi_{n-1}^\dagger$ factor. Finally, we consider the alternative form of the quantum Serre relation \eqref{alt uq Serre rels}\ddichotomy{\relax}{ in \cite{willie_a}} with $(i, j) = (n, n-2)$. Observe that 
	\begin{align*}
		[\widetilde{E}_{n} , \widetilde{E}_{n-2}]_q 
			&= q^{-1} 
				\omega_{n-2}^{-1}
				\left( 
					\psi_{n-1}^\dagger \psi_n^\dagger
					\psi_{n-2}^\dagger \psi_{n-1}^{\pdg}
					-
					q
					\psi_{n-2}^\dagger \psi_{n-1}^{\pdg}
					\psi_{n-1}^\dagger \psi_n^\dagger
				\right) 
			= 
				\omega_{n-1}
				\psi_n^\dagger\psi_{n-2}^\dagger.
	\end{align*}
	Each term in the $\qinv$-commutator $[\widetilde{E}_n, [\widetilde{E}_{n}, \widetilde{E}_{n-2}]_q]_{\qinv}$ contains a factor of $\left(\psi_n^\dagger\right)^2 = 0$, so it vanishes. The relations for the $\widetilde{F}_i$ follow from these by an application of the $*$-structure defined in \crossrefcliff{star struct}.
\end{proof}

We can also embed the odd orthogonal quantum group $\Uqbn$ into $Cl_q(n)$. 
\begin{prop}\label[prop]{uqbn clqn embedding}
	Recall the map $\Phi_{q, n}$ of \crossrefgln{uqgln clqn embedding}. There is an algebra map $\Phi_{q, n}^{\mathbf{B}}\colon \Uqbn \to Cl_q(n)$ satisfying $\Phi_{q, n}^{\mathbf{B}}(X) = \Phi_{q, n}(X)$ for $X$ belonging to the subalgebra $\Uqsln \subset \Uqbn$ and
	\begin{align*}
		\Phi_{q, n}^{\mathbf{B}}(E_n) &= \psi_n^\dagger \\
		\Phi_{q, n}^{\mathbf{B}}(F_n) &= \psi_n^{\pdg} \\
		\Phi_{q, n}^{\mathbf{B}}(K_n) &= q^{1 \over 2} \omega_n^{-1}.
	\end{align*}
\end{prop}
\begin{rmk}
	The quantum Clifford algebra parameter is still $q$, but we extend the base field to include $q^{1 \over 2}$.
\end{rmk}
\begin{proof}
	The claim again follows from a calculation. It suffices to check the relations involving the images $\widetilde{E}_n, \widetilde{F}_n, \widetilde{K}_n$ of $E_n, F_n, K_n$ under $\Phi_{q, n}^{\mathbf{B}}$. In this case let $A = [a_{ij}]$ denote the $n \times n$ Cartan matrix of $\mathfrak{so}_{2n+1}$, as in \Cref{cartan mats}. 
	
	Note that $\widetilde{K}_i \widetilde{E}_n \inv{\widetilde{K}_i} = \widetilde{E}_n$ and similarly $\widetilde{K}_n \widetilde{E}_i \inv{\widetilde{K}_n} = \widetilde{E}_i$ whenever $i < n - 1$. Conversely if $i \geq n-1$, $K_i$ and $E_n$ contain non-commuting factors and
	\begin{align*}
		\widetilde{K}_{n-1} \widetilde{E}_n \inv{\widetilde{K}_{n-1}} 
			= 	\left(\omega_{n-1}^{-1} \omega_{n}\right)
				\psi_{n}^\dagger
				\left(\omega_{n-1}^{-1} \omega_{n}\right)^{-1}
			= \qinv \widetilde{E}_n 
			= q^{{1 \over 2}\cdot d_{n-1} a_{n-1,n}} \widetilde{E}_n.
	\end{align*}
	Similarly, 
	\begin{align*}
		\widetilde{K}_n \widetilde{E}_{n-1} \inv{\widetilde{K}_n} 
			= \qinv \inv{\omega_{n-1}}
			  \omega_n^{-1}
			  (\psi^\dagger_{n-1} \psi_{n}^{\pdg})
			  \omega_n
			= q^{-1} \widetilde{E}_i 
			= q^{{1 \over 2}\cdot d_{n} a_{n,n-1}} \widetilde{E}_n.
	\end{align*}
	Finally, we calculate that
	\begin{align*}
		\widetilde{K}_n \widetilde{E}_n \inv{\widetilde{K}_n} 
			= \omega_{n}^{-1} \psi_{n}^\dagger \omega_{n}
			= q \widetilde{E}_n 
			= q^{{1 \over 2}\cdot d_{n} a_{n,n}} \widetilde{E}_n,
	\end{align*}
	to conclude that $\widetilde{K}_i \widetilde{E}_j \inv{\widetilde{K}_i} = \big(q^{{1 \over 2}}\big)_i^{a_{ij}} \widetilde{E}_j$, with $q_i = q^{d_i}$.
	
	Next, notice that
	\begin{align*}
		[\widetilde{E}_n, \widetilde{F}_n] 
			= [\psi_n^\dagger, \psi_n^{\pdg}] 
			= -{(q^{1 \over 2}\omega_n) - \inv{(q^{1 \over 2}\omega_n)} \over q^{1 \over 2} - q^{-{1 \over 2}}} 
			= {K_n - \inv{K_n} \over q^{1 \over 2} - q^{-{1 \over 2}}}.
	\end{align*}
	
	To finish the proof, we verify the Serre relations. Recall the identity \eqref{alt uq Serre rels}\ddichotomy{\relax}{ in \cite{willie_a}}. First notice that $[\widetilde{E}_n, \widetilde{E}_{j}] = 0$ when $j < n-1$ because $\widetilde{E}_n$ and $\widetilde{E}_j$ do not share common factors. Next we find that
	\begin{align*}
		[\widetilde{E}_{n-1}, \widetilde{E}_n]_{q_{n-1}^{1\over 2}}]
			&= \qinv \omega_{n-1}^{-1}(\psi_{n-1}^\dagger \psi_n^{\pdg}\psi_n^\dagger - q \psi_n^\dagger \psi_{n-1}^\dagger \psi_n)^{\pdg} 
			= \qinv \omega_{n-1}^{-2} \psi_{n-1}^\dagger,
	\end{align*}
	so $[\widetilde{E}_{n-1}, [\widetilde{E}_{n-1}, \widetilde{E}_n]_{q_{n-1}^{1\over 2}}]_{q_{n-1}^{-{1 \over 2}}}$ indeed vanishes because each summand contains a factor of $(\psi_{n-1}^\dagger)^2 = 0$. 
	Finally, since $a_{n, n-1} = -2$, we see that
	\begin{align*}
		\sum_{p=0}^{3} 
			\begin{bmatrix} 3 \\ p \end{bmatrix}_{q^{1/2}} 
			\widetilde{E}_{n}^p \widetilde{E}_{n-1} \widetilde{E}_{n}^{3-p}
		= 0 
	\end{align*}
	because each summand contains a factor of $\widetilde{E}_n^2 = 0$.
	
	The unverified relations involving $\widetilde{F}_n$ follow from the corresponding relations involving $\widetilde{E}_n$ by an application of the $*$-structure defined by \crossrefcliff{star struct}.
\end{proof}

The homomorphisms of \Cref{uqdn clqn embedding,uqbn clqn embedding} immediately yield the decomposition of $\bigwedge_q(\Vn)$ as a $\Uqdn$- and as a $\Uqbn$-module. Since the maps identify $\Uqsln$ as a subalgebra of $\Uqdn$ and $\Uqbn$, respectively, we see that the highest weight vectors with respect to the orthogonal quantum group action are the $\Uqsln$-highest weight vectors that are also annihilated by $E_n$. 

\begin{prop}
	As a $\Uqdn$-module, the braided exterior algebra $\bigwedge_q(\Vn)$ defined by \crossrefgln{braided ext alg as tensor alg quotient} decomposes as 
		$$\textstyle \bigwedge_q(\Vn) \cong S_+ \oplus S_-.$$
	Here $S_\pm$ denote the irreducible $\Uqdn$-modules of highest weight $({1 \over 2}, \ldots, {1 \over 2}, \pm{1 \over 2})$.
\end{prop}
\begin{proof}
The $\Uqdn$ generators $E_i, F_i, K_i$, for $i = 1, \ldots, n-1$ generate the subalgebra $\Uqsln \subset \Uqdn$, so the highest weight vectors of the $\Uqdn$-action are $\Uqgln$-highest weight vectors that are also annihilated by $E_n$. Recall the basis $v(\ell)$ of $\bigwedge_q(\Vn)$ defined by \crossrefgln{braided ext alg basis}. The highest weight vectors with respect to the $\Uqgln$-action are $v(\gamma_j)$, for $j = 0, 1, \ldots, n$, with $\gamma_j = \sum_{k = 0}^j e_j$. Clearly, $v(\ell) \in \ker \mathcal{D}(E_n)$ only if $\ell_{n-1} + \ell_n \geq 1$,
so there are exactly two highest weight vectors:
$$v(1, \ldots, 1, 1) \quad \text{and} \quad v(1, \ldots, 1, 0)$$
of weights $(1/2, \ldots, 1/2, 1/2)$ and $(1/2, \ldots, 1/2, -1/2)$, respectively.	
\end{proof}

We will henceforth refer to the $\Uqdn$-module $S = \bigwedge_q(\Vn)$ as the \textit{quantum spin} module. We note that $S$ may also be equipped with a $\Uqbn$-module structure. A similar kernel calculation shows that $\bigwedge_q(\Vn)$ is irreducible as a $\Uqbn$-module. This is analogous to the classical case: the irreducible $SO(\mathbb{C}^{2n+1})$-module $\bigwedge(\mathbb{C}^n)$ splits as $S_+ \oplus S_-$ when viewed as an $SO(\mathbb{C}^{2n})$-module.
\begin{prop}
	As a $\Uqbn$-module, $\bigwedge_q(\Vn)$ is irreducible. It has a highest weight vector of weight $({1 \over 2}, \ldots, {1 \over 2})$. 
\end{prop}
\begin{proof}
	The proof again relies on \crossrefgln{uqgln ext alg decomp}, which decomposes $\bigwedge_q(\Vn)$ as a $\Uqgln$-module. In this case $v(\ell) \in \ker \mathcal{B}(E_n)$ only if $\ell_n = 1$ so $E_n$ only annihilates a single $\Uqgln$ highest weight vector, namely $v_\mu = v(1, \ldots, 1)$. Its weight with respect to the $\Uqbn$-action is determined by
	\begin{align*}
		\langle \mu, \alpha_i \rangle &= 0, 
		\quad \text{for } i = 1, \ldots n-1, \quad \text{and} \\
		\langle \mu, \alpha_n \rangle &= 1,
	\end{align*}
	since $K_i \rhd v_\mu = q^{{1\over 2} \delta_{in}} v_\mu$. Here $\langle \cdot, \cdot \rangle$ is the bilinear form on the root space normalized so that $\langle \alpha, \alpha \rangle = 2$ for \textit{short} roots, as in \Cref{inner prod on roots}.
	Solving using $\alpha_i = e_i - e_{i+1}$ for $i = 1, \ldots n-1$ and $\alpha_n = e_n$ yields $\mu = ({1 \over 2}, \ldots, {1 \over 2})$, as claimed.
\end{proof}

Although $S$ is \textit{not} irreducible as a $\Uqdn$-module, it is an irreducible representation of the \textit{full orthogonal quantum group} $\Uqodn$. This Hopf algebra is the quantum analogue of the Hopf algebra $U(\on)$ defined in \ref{on defn}.

\begin{dfn}\label[defn]{uqon defn}
	The quantum group $U_q(\mathfrak{o}_n) = U_q(\mathfrak{so}_n) \rtimes \mathbb{Z}_2$ is generated by the enveloping algebra $\Uqson$ and the additional generator $t$, subject to the following relations. If $n$ is odd, then $t$ commutes with every generator. If $n = 2r$ is even, then
	\begin{align}\label{t gen comm}
	\begin{split}
		t E_{r-1} \inv{t} &= E_r, \quad t E_{r} \inv{t} = E_{r-1}, \\
		t F_{r-1} \inv{t} &= F_r, \quad t F_{r} \inv{t} = F_{r-1}, \\
		t K_{r-1} \inv{t} &= K_r, \quad t K_{r} \inv{t} = K_{r-1},
	\end{split}
	\end{align}
	and $t$ commutes with all other generators. The element $t$ is \textit{group-like}, which means we extend the comultiplication $\Delta$ and the antipode $S$ of $U_q(\son)$ to $U_q(\on)$ by specifying that $\Delta(t) = t \otimes t$ and $S(t) = \inv{t}$. 
\end{dfn}

We may extend any $\Uqson$-module to a $U_q(\on)$-module; the $U_q(\on)$-modules are related to the $\Uqson$-modules exactly as in the classical case. \Cref{comm orthog actions on ext alg} explains the relationship.

In particular, $S$ is an irreducible $U_q(\mathfrak{o}_{2n})$-module because $t$ must permute the highest weight spaces of the $U_q(\mathfrak{so}_{2n})$-action. In the odd case, $S$ is already irreducible as a $\Uqbn$-module, and it remains so when we extend the action to $U_q(\mathfrak{o}_{2n+1})$. 

To conclude, we record the decomposition of $S \otimes S$ as a $\Uqodn$-module, for convenience. The decomposition is analogous to the classical situation.

\begin{prop}\label[prop]{spin sq decomp}
	As a $\Uqdn$-module, 
$$S \otimes S \cong \bigoplus_{j = 0}^{2n} Y_j.$$
	Here, $Y_j$ denotes the irreducible $\Uqodn$-module of highest weight $\gamma_j = \sum_{k=0}^j e_k$.
\end{prop}
\begin{proof}
	The decomposition can be computed directly using the Brauer-Klimyk formula. Alternatively, it is \Cref{uqodn uqprime duality decomp} specialized with $m = 2$.
\end{proof}
	
	\subsection{Commuting actions on $S^{\otimes m}$}
	\label{commuting_q_embeddings_bd}
	Recall the construction of the irreducible $U_q(\mathfrak{o}_{2n})$ spin module $S = \bigwedge_q(\Vn)$ described in \Cref{comm orthog actions on ext alg}. In this section we study the commutant of the quantum group action on the tensor product $S^{\otimes m}$. We begin by showing that $S^{\otimes m} \cong \bigwedge_q(V^{(nm)})$ as a $U_q(\mathfrak{so}_{2n})$-module, and we show that the action of $U_q(\mathfrak{so}_{2n})$ on $S^{\otimes m}$ factors through the quantum Clifford algebra $Cl_q(nm)$. We then use the construction of \crossrefgln{q skew duality type a} to obtain an action of the co-ideal subalgebra $\Uqprime \subset \Uqglm$ on $\bigwedge_q(V^{(nm)})$ that also factors through $Cl_q(nm)$. This action commutes with that of the subalgebra $\Uqsln \subset \Uqdn$ automatically by \crossrefgln{uqgln uqglm embeddings commute}. We conclude by showing that the $\Uqprime$ on $S^{\otimes m} \cong \bigwedge_q(V^{(nm)})$ indeed commutes with whole the $U_q(\mathfrak{so}_{2n})$ action.

Motivated by the classical embedding of \Cref{so2n into cl2n embedding}, we begin by constructing an algebra map $\mathcal{L}_q\colon \Uqdn \to Cl_q(nm)$ that extends the homomorphism $\lambda_q \colon \Uqgln \to Cl_q(nm)$ defined in \crossrefgln{uqgln clqnm embedding}. This map equips the braided exterior algebra $\bigwedge_q(V^{(nm)})$ with a $\Uqdn$-module structure making the following diagram commute.
\begin{equation}\label[diag]{uqson action diag}
	\begin{tikzcd}[column sep=1.45cm, row sep=1.1cm]
		\Uqsln \dar \ar[r, "\widetilde{\Delta}^{(m-1)}"] 
		& \Uqsln^{\otimes m} \dar \ar[r, "\Phi_{q, n}^{\otimes m}"]
		& Cl_q(n)^{\otimes m} \ar[r, "\pi_0^{\otimes m}"] \ar[d, "\id"]
		& \End\left(\bigwedge_q(\Vn)^{\otimes m} \right) \ar[d, "\id"] \\
		\Uqdn \ar[r, "\widetilde{\Delta}^{(m-1)}"] \ar[drr, dashed, blue!60, "\mathcal{L}_q"]
		& \Uqdn^{\otimes m} \ar[r, "\big(\Phi_{q, n}^{\mathbf{D}}\big)^{\otimes m}"]
		& Cl_q(n)^{\otimes m} \ar[r, "\pi_0^{\otimes m}"] \ar[d, "\Gamma_q"] 
		& \End\left(S^{\otimes m}\right) \dar \\
		\Uqsln \ar[rr, swap, "\lambda_q"] \uar
		& 
		& Cl_q(nm) \ar[r, "\pi_0"]
		& \End\left(\bigwedge_q(V^{(nm)})\right) 
	\end{tikzcd}
\end{equation} 
Here $\widetilde{\Delta}$ denotes the comultiplication map satisfying
\begin{align}\label{backward comult so review}
	\begin{split}
		\widetilde{\Delta}(E_i) &= E_i \otimes 1 + K_i \otimes E_i, \\
		\widetilde{\Delta}(F_i) &= F_i \otimes \kinv + 1 \otimes F_i, 
		\quad \text{and}\\
		\widetilde{\Delta}(L_i) &= L_i \otimes L_i\ddichotomy{\relax}{.}
	\end{split}
\end{align}
as in \crossrefgln{backward comult}. Recall \Cref{uqdn clqn embedding} defines $\Phi_{q, n}^{\mathbf{D}}$. We take $\Gamma_q$ and $\pi_0$ as in \crossrefcliff{clqn clqnm embedding,clqnk repn} and $\lambda_q$ and $\Phi_{q, n}$ as in \crossrefgln{uqgln clqn embedding,uqgln clqnm embedding}.

\begin{prop}\label[prop]{uqdn clqnm embedding}
	Recall the notation of \Cref{uqgln clqnm embedding}. In addition, define
	$$
	\kappa_{n, <j} = \prod_{p = 0}^{j-1} \left(q \omega_{n-1 + pn} \omega_{n + pn}\right)^{-1}
	\quad \text{and} \quad 
	\kappa_{n, >j} = \prod_{p = j}^{m-1} \left(q \omega_{n-1 + pn} \omega_{n + pn}\right)^{-1},
	$$
	by taking an appropriate product of $\omega_a$ generators in the $(n-1)$st and $n$th rows to the left, and respectively to the right, of the $j$th column.
	
	There is an algebra homomorphism $\mathcal{L}_q\colon \Uqdn \to Cl_q(nm)$ satisfying
	\begin{align*}
		E_n &\to \sum_{j=0}^{m-1} \psi^\dagger_{n-1 + jn} \psi_{n + jn}^\dagger \kappa_{n, <j}, \\
		F_n &\to \sum_{j=0}^{m-1} \kappa_{n, >j}^{-1} \psi_{n + jn}^{\pdg}\psi_{n-1 + jn}^{\pdg}, \text{ and} \\
		K_n &\to \kappa_{n, <m}
	\end{align*}
	and $\mathcal{L}_q(X) = \lambda_q(X)$ for every $X$ belonging to the subalgebra $\Uqsln \subset \Uqdn$.
\end{prop}
\begin{proof}
Notice that $\mathcal{L}_q = \Gamma_q \circ \mathcal{D}_n^{\otimes m} \circ \widetilde{\Delta}^{(m-1)}$ is a composition of algebra maps.
\end{proof}

The commutativity of \Cref{uqson action diag} immediately yields an isomorphism of $\Uqdn$-modules between the $m$-fold tensor product of the quantum spin module $S$ and the braided exterior algebra $\bigwedge_q(V^{(nm)})$.

\begin{cor}\label[cor]{sm cong braided ext alg}
	As $\Uqdn$-module, $S^{\otimes m} \cong \bigwedge_q(V^{(nm)})$.
\end{cor}
\begin{rmk}
	Notice that $S^{\otimes m}$ is no longer a $\Uqdn$-module algebra in the sense of \cite[Definition~4.1.1]{montgomery_1993}, unless we deform the underlying algebra structure of $\bigwedge_q(V^{(nm)})$ using a twisted multiplication as in \cite[Theorem~2.3]{lzz_2010}.\longer{ For convenience, \Cref{module alg defn} recalls the definition of a $\Uqg$-\textit{module algebra}.}
\end{rmk}

Notice $\mathcal{L}_q$ maps each $K_i$ to a diagonal operator with respect to the $v(\ell)$ basis of $\bigwedge_q(V^{(nm)})$ defined by \crossrefgln{braided ext alg basis}. Later we will need the weight of each $v(\ell)$ with respect to the action of $\Uqdn$, so we compute it in the next lemma. 

Recall that each $v(\ell)$ determines a state of occupied and vacant positions in an $n \times m$ grid corresponding to the following arrangement of weight vectors in $V^{(nm)}$. 
\begin{equation}\label{basis rect diag orthog review}
	\begin{array}{ccccc}
	v_{11} & v_{12} & v_{13} & \cdots & v_{1m} \\
	v_{21} & v_{22} & v_{23} & \cdots & v_{2m} \\
	\vdots & \vdots & \vdots & \ddots & \vdots \\
	v_{n1} & v_{n2} & v_{n3} & \cdots & v_{nm}
	\end{array}
\end{equation}
Here we use the shorthand $v_{ij} = v_{i + (j-1)n}$. The $i$th component of $v(\ell)$'s weight is determined by the state of the $i$th row in the rectangular array: each occupied position adds $+1/2$ while each vacant position contributes $-1/2$.

\begin{lem}\label[lem]{uqdn weight on basis}
	For each $\ell \in \{0,1\}^{nm}$, let $v(\ell)$ denote the basis vector of $\bigwedge_q(V^{(nm)})$ defined by \crossrefgln{braided ext alg basis}. Its weight with respect to the action of $\Uqdn$ via $\mathcal{L}_q$ is 
	$$\ell \mathbf{1}_m - {m \over 2} \mathbf{1}_n.$$
	Here $\ell$ is viewed as an $n \times m$ matrix and $\mathbf{1}_p$ denotes the $p$-vector of all ones.
\end{lem}
\begin{proof}
	The weight $\mu$ is uniquely determined by the equations
	\begin{align*}
		q^{\langle \mu, \alpha_i \rangle} v(\ell) 
		= \mathcal{L}_q(K_i) v(\ell) 
		= \prod_{j=0}^{m-1} \omega_{i + jn}^{-1} \omega_{i+1 + jn} v(\ell)
		= q^{\sum_{j=1}^m \ell_{ij} - \sum_{j=1} \ell_{i+1, j}}v(\ell)
	\end{align*}
	for $i = 1, \ldots, n-1$, and 
	\begin{align*}
		q^{\langle \mu, \alpha_n \rangle} v(\ell) 
		= \mathcal{L}_q(K_i) v(\ell) 
		= \prod_{j=0}^{m-1} (q\omega_{n-1 + jn}\omega_{n + jn})^{-1} v(\ell)
		= q^{-m + \sum_{j=1}^m \ell_{n-1,j} + \sum_{j=1} \ell_{nj}}v(\ell).
	\end{align*}
	Combining equations implies that 
	\begin{align*}
		\mu_{n} 
		&= {1 \over 2} 
		  \langle \mu, \alpha_n - \alpha_{n-1} \rangle 
		= -{m \over 2} + \sum_{j=1}^m \ell_{nj}. 
	\end{align*}
	Backward substitution into the remaining equations yields the desired result.
\end{proof}

We now turn to studying the commutant of the $\Uqdn$-action on $S^{\otimes m} \cong \bigwedge_q(V^{(nm)})$. The classical duality result for orthogonal groups relies on the embedding $O(\mathbb{C}^m) \subset GL(\mathbb{C}^m)$ at the Lie group level, which induces a map $\mathfrak{so}_m \hookrightarrow \mathfrak{gl}_m$ at the level of Lie algebras. The last map has no quantum analogue. In other words, there is no algebra homomorphism $U_q(\mathfrak{so}_m) \to \Uqglm$. 

Notwithstanding, the commutant of the $\Uqdn$ action on $S^{\otimes m}$ is generated by the \textit{non-standard deformation} $\Uqprime$ of the Lie algebra $\mathfrak{so}_m$, which can be realized as a co-ideal subalgebra of $\Uqglm$. Realizing $\Uqprime$ as a subalgebra of $\Uqglm$ allows us to use our skew $\Uqgln \otimes \Uqglm$-duality result \crossrefgln{uqgln uqglm duality} for Type $\mathbf{A}$ in order to compute the centralizer of $\Uqdn$ in $\End(S^{\otimes m})$.

\begin{dfn}
	The \textit{non-standard deformation} $\Uqprime$ of $\mathfrak{so}_m$ is the unital associative algebra generated by $B_1, \ldots, B_{m-1}$ subject to the relations $B_i B_j = B_j B_i$, whenever $i \neq j$, and
\begin{equation}\label{uqprime def rel}
B_j^2 B_{j\pm 1} - (q + \inv{q})B_j B_{j\pm 1} B_j + B_{j\pm 1} B_j^2 = B_{j \pm 1}.
\end{equation}
\end{dfn}
\begin{rmk}
	\Cref{uqprime def rel} is sometimes referred to as the \textit{$q$-Serre relation}. In particular, see \cite{rowell_wenzl_2017}. Note that it differs from the Serre \ddichotomy{\cref{uq Serre rels}}{relation} defining a Drinfeld-Jimbo quantum group in that the right hand side is \textit{not} identically zero.
\end{rmk}

\begin{prop}\label[prop]{uqsom prime as a coideal subalg}
	Let $E_j, F_j, K_j$, for $j = 1, \ldots, m-1$ denote generators for $U_q(\mathfrak{sl}_m) \subset \Uqglm$. The algebra $\Uqprime$ can be realized as a left co-ideal subalgebra of $\Uqglm$ by setting
	$$B_j = \sqrt{-1} (F_j - q \inv{K_j} E_{j}).$$
\end{prop}
\begin{proof}
	First notice that the algebra generated by the $B_j$ is a co-ideal subalgebra of $\Uqglm$: for any $j = 1, \ldots, m-1$,
	\begin{align*}
	-\sqrt{-1} \, \Delta(B_j) &= F_j \otimes 1 + \inv{K_j} \otimes F_j - q(\inv{K_j} \otimes  \inv{K_j}) (E_j \otimes  K_j + 1 \otimes E_j)\\
	&= B_j \otimes 1 + \inv{K_j} \otimes B_j 
	\in U_q \mathfrak{sl}_m \otimes U_q'\mathfrak{so}_m.
	\end{align*}
	
	Next observe that if $|i - j| > 1$ then $[E_i, E_j] = [F_i, F_j] = 0$ by the quantum Serre relations in $\Uqglm$, so $[B_i, B_j] = 0$ as well. 
	
	It remains to show that the $B_j$ satisfy \Cref{uqprime def rel}. This follows from a straightforward yet somewhat tedious and unenlightening calculation that is left to the reader as an exercise.

\end{proof}

Next we embed $\Uqprime$ into $Cl_q(nm)$. We have done all the heavy lifting already: we need only combine \Cref{uqsom prime as a coideal subalg}, which expresses each generator of $\Uqprime$ in terms of $\Uqglm$ generators, with \crossrefgln{uqglm clqnm embedding}, which maps $\Uqglm$ into $Cl_q(nm)$. The resulting map equips the $Cl_q(nm)$-module $S^{\otimes m} \cong \bigwedge_q(V^{(nm)})$ with a $\Uqprime$-action.
\begin{prop}\label[prop]{uqprime clqnm embedding}
	Recall the map $\rho_q\colon \Uqglm \to Cl_q(nm)$ defined in \crossrefgln{uqglm clqnm embedding}. There is an algebra map $\mathcal{R}_q\colon \Uqprime \to Cl_q(nm)$ resulting from the composition
	\begin{equation*}
	\begin{tikzcd}
		\Uqprime \rar
		& \Uqglm \ar[r, "\rho_q"]
		& Cl_q(nm)
	\end{tikzcd}	
	\end{equation*}
\end{prop}

We conclude this section by showing that $\mathcal{L}_q\colon \Uqdn \to Cl_q(nm)$ and $\mathcal{R}_q\colon \Uqprime \to Cl_q(nm)$ induce commuting actions on the $Cl_q(nm)$-module $S^{\otimes m}$.
\begin{prop}\label[prop]{uqdn uqprime commute}
	The embeddings $\mathcal{L}_q$ and $\mathcal{R}_q$ defined in \Cref{uqdn clqnm embedding,uqprime clqnm embedding} induce commuting subalgebras of $\End\left(\bigwedge_q(V^{(nm)})\right)$.
\end{prop}
\begin{proof}
	The proof follows from a calculation. In \crossrefgln{uqgln uqglm embeddings commute} we show $\Uqgln \subset \Uqdn$ and $\Uqglm$ induce commuting actions on $\bigwedge_q(V^{(nm)})$. \Cref{uqsom prime as a coideal subalg} realizes $\Uqprime$ as a subalgebra of $\Uqglm$, so it suffices to show 
	$$
		[\mathcal{L}_q(K_n^{(2n)}), \mathcal{R}_q(B_i)] = 
		[\mathcal{L}_q(E_n^{(2n)}), \mathcal{R}_q(B_i)] = 
		[\mathcal{L}_q(F_n^{(2n)}), \mathcal{R}_q(B_i)] = 
		0.
	$$
	We use superscripts on quantum group generators\longer{ as in \Cref{generator superscript}} to indicate the algebra to which they belong; for instance $E_n^{(2n)}$ belongs to $\Uqdn$ while $F_i^{(m)}$ lives in $\Uqglm$.
	
	First notice that
	\begin{align*}
		\mathcal{L}_q(K_n^{(2n)}) &\mathcal{R}_q(E_j^{(m)}) \mathcal{L}_q(K_n^{(2n)})^{-1}\\
			&= \sum_{b=1}^n \kappa_{>b, j} 
				\left(
					\prod_{a=0}^{m-1} \omega_{n-1 + an}^{-1} \omega_{n + an}^{-1}
				\right)
				\psi_{b + (j-1)n}^\dagger \psi_{b + jn}^{\pdg}
				\left(
					\prod_{a=0}^{m-1} \omega_{n-1 + an} \omega_{n + an}
				\right) \\
			&= \mathcal{R}_q(E_j^{(m)}).
	\end{align*}
	We also have $\mathcal{L}_q(K_n^{(2n)}) \mathcal{R}_q(F_j^{(m)}) \mathcal{L}_q(K_n^{(2n)})^{-1} = \mathcal{R}_q(F_j^{(m)})$ for each $j = 1, \ldots, m-1$, so $[\mathcal{L}_q(K_n^{(2n)}), \, \mathcal{R}_q(B_j^{(m)})] = 0$.
	
	Next we compute $[\mathcal{L}_q(E_n^{(2n)}), \mathcal{R}_q(B_j^{(m)})]$ in two steps. For each $j = 1, \ldots, m-1$, let $A_j$ denote the \textit{central} anticommutator $\{\psi_{n-1 + (j-1)n}^\dagger, \psi_{n-1 + (j-1)n}^{\pdg}\}$. By sliding all $\omega_a$ generators to the left of a product and rearranging the $\phi_a$ generators when possible, we see all terms in the commutator $[\mathcal{L}_q(E_n^{(2n)}), \mathcal{R}_q(F_j^{(m)})]$ vanish except for two:
	\begin{align*}
		[&\mathcal{L}_q(E_n^{(2n)}), \, \mathcal{R}_q(F_j^{(m)})] \\
		&= 
		\sum_{a=1}^m \sum_{b \leq n-1} 
		\kappa_{n, <a} \kappa_{<j, b}^{-1}
		\psi^\dagger_{n-1 + (a-1)n} \psi_{n + (a-1)n}^\dagger 
		\psi_{b + jn}^\dagger \psi_{b + (j-1)n}^{\pdg}\\
		&\quad -
		\sum_{a = 1}^m \sum_{b \leq n-2}
		\kappa_{n, <a}
		\kappa_{<j, b}^{-1}
		\psi^\dagger_{n-1 + (a-1)n} \psi_{n + (a-1)n}^\dagger 
		\psi_{b + jn}^\dagger \psi_{b + (j-1)n}^{\pdg} \\
		&\quad -
		\sum_{\substack{b = n-1 \\ a \neq j, j+1}} 
		\kappa_{n, <a}
		\kappa_{<j, b}^{-1}
		\psi^\dagger_{n-1 + (a-1)n} \psi_{n + (a-1)n}^\dagger 
		\psi_{b + jn}^\dagger \psi_{b + (j-1)n}^{\pdg} \\
		&\quad +
		\sum_{a, b=n}
		q^{\delta_{aj} - \delta_{a,j+1}}
		\kappa_{n, <a}
		\kappa_{<j, b}^{-1}
		\psi^\dagger_{n-1 + (a-1)n} \psi_{n + (a-1)n}^\dagger 
		\psi_{b + jn}^\dagger \psi_{b + (j-1)n}^{\pdg} \\
		&\quad -
		\sum_{\substack{b = n \\ a \neq j, j+1}} 
		\kappa_{n, <a}
		\kappa_{<j, b}^{-1}
		\psi^\dagger_{n-1 + (a-1)n} \psi_{n + (a-1)n}^\dagger 
		\psi_{b + jn}^\dagger \psi_{b + (j-1)n}^{\pdg} \\
		&\quad + 
		\kappa_{n, <j} \inv{\kappa_{< n-1, j}}
		\left(
			\psi_{n-1 +(j-1)n}^\dagger \psi_{n + (j-1)n}^\dagger 
			\psi_{n + jn}^\dagger \psi_{n-1 + (j-1)n}^{\pdg}
			\right. \\
		  &\qquad -
			\left.
			\psi_{n + jn}^\dagger \psi_{n-1 + (j-1)n}^{\pdg}
			\psi_{n-1 +(j-1)n}^\dagger \psi_{n + (j-1)n}^\dagger 
		\right) \\[+3pt]
		&\quad + 
		\kappa_{n, <j}\inv{\kappa_{<n, j}}
		\left(
			q \psi_{n-1 +(j-1)n}^\dagger \psi_{n + (j-1)n}^\dagger 
			\psi_{n + jn}^\dagger \psi_{n + (j-1)n}^{\pdg}
			\right.\\
		  &\qquad -	
			\left.
			\psi_{n + jn}^\dagger \psi_{n + (j-1)n}^{\pdg}
			\psi_{n-1 +(j-1)n}^\dagger \psi_{n + (j-1)n}^\dagger 
		\right) \\[+3pt]
		&\quad - 
		\kappa_{n, < j+1} \inv{\kappa_{< n-1, j}}
		\left(
			\psi_{n + jn}^\dagger (\psi_{n-1 + jn}^\dagger)^2 \psi_{n-1 + (j-1)n}^{\pdg}
		-
			\qinv \psi_{n-1 + (j-1)n}^{\pdg} (\psi_{n-1 + jn}^\dagger)^2 \psi_{n + jn}^\dagger
		\right) \\[+3pt]
		&\quad +
		\kappa_{n, < j+1} \inv{\kappa_{<n, j}}
		\left(
			\qinv \psi_{n-1 + jn}^\dagger (\psi_{n + jn}^\dagger)^2 \psi_{n +(j-1)n}^{\pdg}
			q \psi_{n +(j-1)n}^{\pdg} (\psi_{n + jn}^\dagger)^2 \psi_{n-1 + jn}^\dagger
		\right) \\[+3pt]
		&= 
		\kappa_{n, <j}
		\left(
			\inv{\kappa_{<n,j}} \omega_{n + (j-1)n}^{-1} 
			\cdot
			\psi_{n +jn}^\dagger \psi_{n-1 + (j-1)n}^\dagger
		-
			\inv{\kappa_{<n-1,j}} A_j 
			\cdot 
			\psi_{n-1 +jn}^\dagger \psi_{n + (j-1)n}^\dagger
		\right).
	\end{align*}
	Similarly, all summands in $[\mathcal{L}_q(E_n^{(2n)}), \mathcal{R}_q(\inv{(K_j^{(m)})} E_j^{(m)})]$ vanish except for two:
	\begin{align*}
		&[\mathcal{L}_q(E_n), \,\mathcal{R}_q(\inv{(K_j^{(m)})} E_j^{(m)})] \\
		&=
		\mathcal{R}_q\left(K_j^{(m)}\right)^{-1} 
		\left(
			\sum_{a=1}^m \sum_{b \leq n-2}
			\kappa_{n, <a} 
			\kappa_{>b, j}
			\psi_{n-1 + (a-1)n}^\dagger \psi_{n + (a-1)n}^\dagger
			\psi_{b + (j-1)n}^\dagger \psi_{b + jn}^{\pdg}
			\right. \\
		&\quad - 
			\sum_{a=1}^m \sum_{b \leq n-2}
			\kappa_{n, <a} 
			\kappa_{>b, j}
			\psi_{n-1 + (a-1)n}^\dagger \psi_{n + (a-1)n}^\dagger
			\psi_{b + (j-1)n}^\dagger \psi_{b + jn}^{\pdg}\\
		&\quad + 
			\sum_{a, b=n-1}
			q^{\delta_{aj} - \delta_{a, j+1}}
			\kappa_{n, <a} 
			\kappa_{>b, j}
			\psi_{n-1 + (a-1)n}^\dagger \psi_{n + (a-1)n}^\dagger
			\psi_{b + (j-1)n}^\dagger \psi_{b + jn}^{\pdg}\\
		&\quad -
			\sum_{\substack{b=n-1 \\ a \neq j, j+1}}
			\kappa_{n, <a} 
			\kappa_{>b, j}
			\psi_{n-1 + (a-1)n}^\dagger \psi_{n + (a-1)n}^\dagger
			\psi_{b + (j-1)n}^\dagger \psi_{b + jn}^{\pdg}\\
		&\quad +
			\sum_{\substack{b=n \\ a \neq j, j+1}}
			q^{2\delta_{aj} - 2\delta_{a, j+1}}
			\kappa_{n, <a} 
			\psi_{n-1 + (a-1)n}^\dagger \psi_{n + (a-1)n}^\dagger
			\psi_{b + (j-1)n}^\dagger \psi_{b + jn}^{\pdg} \\
		&\quad -
			\sum_{\substack{b=n \\ a \neq j, j+1}}
			\kappa_{n, <a} 
			\psi_{n-1 + (a-1)n}^\dagger \psi_{n + (a-1)n}^\dagger
			\psi_{b + (j-1)n}^\dagger \psi_{b + jn}^{\pdg}\\
		&\quad -
			\kappa_{n, <j} \kappa_{>n-1, j} 
			\left(\psi_{n-1 +(j-1)n}^\dagger \right)^2
			\left(
				q \psi_{n + (j-1)n}^\dagger \psi_{n-1 + jn}^{\pdg}
			- 
				\psi_{n-1 + jn} \psi_{n + (j-1)n}^\dagger
			\right) \\[+3pt]
		&\quad +
			\kappa_{n, <j}
			\left(\psi_{n + (j-1)n}^\dagger \right)^2
			\left(
				q^2 \psi_{n-1 + (j-1)n}^\dagger \psi_{n + jn}^{\pdg}
			-
				\psi_{n + jn} \psi_{n-1 + (j-1)n}^\dagger
			\right) \\[+3pt]
		&\quad - 
			\qinv
			\kappa_{n, <j+1} \kappa_{>n-1, j}
			\psi_{n-1 + (j-1)n}^\dagger
			\left(
				\psi_{n-1 + jn}^\dagger
				\psi_{n-1 + jn}^{\pdg}
			+
				\psi_{n-1 + jn}^{\pdg}
				\psi_{n-1 + jn}^\dagger  
			\right)
			\psi_{n + jn}^\dagger \\[+3pt]
		&\quad  +
			\kappa_{n, <j+1}
			\psi_{n + (j-1)n}^\dagger
			\left(
				q^{-2}
				\psi_{n + jn}^\dagger
				\psi_{n + jn}^{\pdg}
			 + \,
				\qinv
				\psi_{n + jn}^{\pdg}
				\psi_{n + jn}^\dagger
			\right) 
			\psi_{n-1 + jn}^\dagger
		\bigg) \\[+3pt]
		&= 
		\qinv \widetilde{K}_j^{-1} \kappa_{n, <j+1} \omega_{n + jn} 
		\left(
			\psi_{n + (j-1)n}^\dagger \psi_{n-1 + jn}^\dagger
		-
			\omega_{n + (j-1)n}^{-1} A_{j+1} 
			\cdot 
			\psi_{n-1 + (j-1)n}^\dagger \psi_{n + jn}^\dagger
		\right).
	\end{align*}
	In the last equality $\widetilde{K}_j = \mathcal{R}_q(K_j^{(m)})$. Now since $\mathcal{R}_q\big(K_j^{(m)}\big)^{-1} = \kappa_{<n, j} \cdot \omega_{n + (j-1)n} \omega_{n + jn}$, we may combine like terms to obtain 
	\begin{align}\label{last en bj calc}
	\begin{split}
		[&\mathcal{L}_q(E_n^{(2n)}), \, \mathcal{R}_q(B_j)] 
		= 
		[\mathcal{L}_q(E_n^{(2n)}), \mathcal{R}_q(F_j^{(m)})] 
		- q [\mathcal{L}_q(E_n^{(2n)}), \mathcal{R}_q\big(\inv{(K_j^{(m)})} E_j^{(m)}\big)] \\
		&= 
		\kappa_{n, <j} 
		\left(
			\inv{\kappa_{< n, j}} \omega_{n + (j-1)n}^{-1} 
			(1 - \qinv \omega_{n-1 + (j-1)n} A_{j+1})
			\cdot
			\psi_{n + jn}^\dagger \psi_{n-1 + (j-1)n}^\dagger
			\right. \\
		  &\qquad +
		  	\left.
		  	\inv{\kappa_{< n-1, j}}
			(A_{j} - \qinv \omega_{n-1 + jn}^{-1}) 
			\cdot
			\psi_{n + (j-1)n}^\dagger \psi_{n-1 + jn}^\dagger 
		\right).
	\end{split}
	\end{align}
	Every anticommutator $\{\psi_a^\dagger, \psi_a\}$ is central, so it acts as the identity on $\bigwedge_q(V^{(nm)})$ \crossrefcliff{anticomm are central}. This means $(1 - \qinv \omega_{n-1 + (j-1)n} A_{j+1})\psi_{n-1 + (j-1)n}^\dagger$ and $(A_j - \qinv \omega_{n-1 + jn}^{-1}) \psi_{n-1 + jn}^\dagger$ act as zero. Hence \Cref{last en bj calc} implies $\mathcal{L}_q(E_n^{(2n)})$ and $\mathcal{R}_q(B_j)$ induce commuting module endomorphisms. 
	
	It remains to check that $\mathcal{L}_q(F_n^{(2n)})$ and $\mathcal{R}_q(B_j)$ induce commuting endomorphisms. This follows from a calculation similar to the one above and it is left as an exercise for the reader. All of these computations may be verified using {\sc SageMath}.
\end{proof}

	\subsection{Multiplicity-free decomposition of $S^{\otimes m}$}
	\label{q_mf_decomposition_bd}
	In the previous section we found homomorphic images of $\Uqdn$ and $\Uqprime$ in the quantum Clifford algebra $Cl_q(nm)$. These maps generate commuting subalgebras of $\End\big(\bigwedge_q(V^{(nm)})\big)$. Recall \Cref{commuting_q_embeddings_bd} extends the $\Uqdn$-module $S$ into an irreducible $\Uqodn$-module. It turns out that $S^{\otimes m} \cong \bigwedge_q(V^{(nm)})$ is irreducible as a $\Uqodn \otimes \Uqprime$-module and the commuting endomorphism algebras generated by $\Uqodn$ and $\Uqprime$ are in fact each others commutants. This is \Cref{uqodn uqprime duality}. 

Before we state our skew $\Uqodn \otimes \Uqprime$ duality theorem, we briefly discuss the representation theory of $\Uqprime$. For the rest of this section, we assume $m = 2r$ or $m = 2r + 1$. There is a notion of roots and weights of $\Uqprime$ and we may identify them with vectors in $\mathbb{R}^{r}$, as usual. Here the analogue of the Cartan subalgebra is the algebra $\mathfrak{h}'$ generated by $B_1, B_3, \ldots, B_{2r - 1}$. A vector $v$ in a $\Uqprime$-module is said to have \textit{weight} $\mu$ if 
\begin{align*}
	B_{2j-1} v = [\mu_j] v
\end{align*}
for all $j = 1, \ldots, r$. As usual, $[n] = (q^n - q^{-n}) / (q - \qinv)$ denotes a $q$-integer. 

\begin{thm}{\cite{gavrilik_klimyk_1991}}
	For each dominant $\mathfrak{so}_m$ weight $\mu$, there exists a finite dimensional simple $\Uqprime$-module $W_\mu$ with highest weight $\mu$ and the same weight multiplicities as the corresponding $U(\mathfrak{so}_m)$-module.
\end{thm}

The finite-dimensional $\Uqprime$-module $W_\mu$ may be realized explicitly using a basis indexed by Gelfand-Tsetlin patterns as in the classical case \cite{gavrilik_klimyk_1991}. This means $v_\mu$ is a \textit{highest weight vector} in $V_\mu$ if it is a weight vector of dominant weight $\mu$.


\begin{thm}\label{uqodn uqprime duality}
	Let $S \cong S_+ \oplus S_-$ denote the so-called $\Uqdn$ ``spin'' module. As a $\Uqodn \otimes \Uqprime$-module, $S^{\otimes m} \cong \bigwedge_q(V^{(nm)})$ is multiplicity-free. In particular,
	\begin{align}\label{uqodn uqprime duality decomp}
		S^{\otimes m} \cong \bigoplus_\mu V_\mu \otimes W_{\bar{\mu}}
	\end{align}
	as a $\Uqodn \otimes \Uqprime$-module. The sum ranges over all partitions $\mu$ fitting in a $(2n) \times r$ rectangle, $V_\mu$ denotes the irreducible $\Uqodn$-module with highest weight and $\mu$, $W_\nu$ denotes the irreducible $\Uqprime$-module indexed by $\nu$, and we take $\bar{\mu}$ as in \Cref{mu bar}.
	
	Consequently, $\Uqodn$ and $\Uqprime$ generate mutual commutants in $\End\left(S^{\otimes m}\right)$.
\end{thm}

Much like \crossrefgln{uqgln uqglm duality} proves that the irreducibles appearing in the decomposition of the braided exterior algebra $\bigwedge_q(V^{(nm)})$ as a $\Uqgln \otimes \Uqglm$-module are parametrized by the same weights as the irreducible components in the decomposition of $\bigwedge(\mathbb{C}^n \otimes \mathbb{C}^m)$ as a $\gln \otimes \glm$-module, \Cref{uqodn uqprime duality} shows that the $\Uqodn \otimes \Uqprime$-isotypic components in the decomposition of $S^{\otimes m} \cong \bigwedge_q(V^{(nm)})$ are labeled by the same weights as the irreducibles in the decomposition of $\bigwedge(\mathbb{C}^n \otimes \mathbb{C}^m)$ as a $U(\mathfrak{o}_{2n}) \otimes U(\mathfrak{so}_m)$-module.

To prove \Cref{uqodn uqprime duality}, we argue as in our proof of our skew duality result for Type $\mathbf{A}$ \crossrefgln{uqgln uqglm duality}. The first step is to identify, for each component in the decomposition of \Cref{uqodn uqprime duality decomp}, an element of $S^{\otimes m} \cong \bigwedge_q(V^{(nm)})$ that is a highest weight vector with respect to the joint $\Uqdn \otimes \Uqprime$-action. Keeping in mind the relationship between $\Uqodn$ and $\Uqdn$-modules, as described in \Cref{q orthog duality}, we then use a dimension count and appeal to the classical result to conclude.

Our construction of joint $\Uqdn \otimes \Uqprime$-highest weight vectors begins with the highest weight vectors with respect to the $\Uqgln \otimes \Uqglm$-action constructed in \crossrefgln{uqgln uqglm hwv}: for each partition $\nu$ with at most $n$ parts satisfying $\nu_1 \leq m$, the corresponding highest weight vector $v_\nu$ is the product of all basis vectors in \Cref{basis rect diag orthog review} that fit in the Young diagram defined by $\nu$.

Now consider a partition $\mu$ with at most $n$ parts satisfying $\mu \leq r$. Notice $\nu$ can have twice as many columns as $\mu$. For each  $\nu = (\nu_1, \ldots, \nu_n)$ with $0 \leq \nu_i \leq m$, let
$$
	v_\nu^* 
	= \left(v_{1 + (m - 1)n} \cdots v_{1 + (m - \nu_1 - 1)n}\right)
	\cdots 
	\left(v_{1 + (m - 1)n} \cdots v_{1 + (m - \nu_n - 1)n}\right).
$$
As illustrated in \Cref{partition reversal}, $v_\nu^*$ determines the state obtained by reflecting the state determined by $v\big(\sum_{i=1}^n \sum_{j=1}^{\nu_i} e_{i + (j-1)n}\big)$ across the vertical mid-segment. 

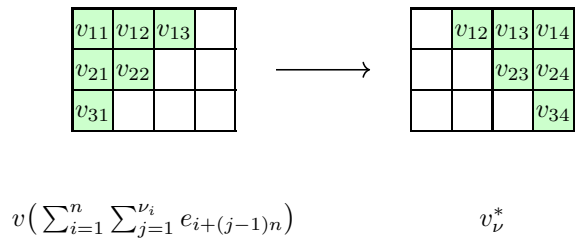
\begin{figure}[!h]
	\begin{center}
	\begin{tikzpicture}
	\node at (0, \yy) {\begin{varwidth}{5cm}{
		\begin{ytableau}
			*(green!20) v_{11} 
			& *(green!20) v_{12} 
			& *(green!20) v_{13} 
			& \\
			*(green!20) v_{21} 
			& *(green!20) v_{22} 
			& 
			& \\
			*(green!20) v_{31} 
			& 
			&
			& \\ 
		\end{ytableau}}\end{varwidth}};
	\node at (0, 0) {$v\big(\sum_{i=1}^n \sum_{j=1}^{\nu_i} e_{i + (j-1)n}\big)$};
	\node at (4.5, \yy) {\begin{varwidth}{5cm}{
		\begin{ytableau}
			\, & *(green!20) v_{12} 
			& *(green!20) v_{13} 
			& *(green!20) v_{14} \\
			&
			& *(green!20) v_{23} 
			& *(green!20) v_{24} \\
			& 
			&
			& *(green!20) v_{34} 
		\end{ytableau}}\end{varwidth}};
	\node at (4.5, 0) {$v_\nu^*$};
	\node (C) at (1.5, \yy) {};
	\node (D) at (3, \yy) {};
	\draw[->,semithick] (C) edge (D);
	\end{tikzpicture}	
	\end{center}
	\caption{Reflecting the state determined by $\nu$ to obtain $v_\nu^*$. Here $(n, m) = (3, 4)$ and $\nu = (3, 2, 1)$.}
	\label{partition reversal}
\end{figure}
For each $p = (p_1, \ldots, p_{r})$, with $0 \leq p_j \leq n$, define 
\begin{align*}
	w^{p} = \prod_{j=1}^{r} \left(\prod_{a=1}^{p_j} v_{n-a + 2(j - 1)n}\right)
\end{align*}
The blue boxes in \Cref{orthog hwv} illustrate the vectors $w^p$ when $(n, m) = (3, 4)$ and $p = (2, 1)$. Let $\#_k(p)$ denote the number of entries in $p$ that are at least $k$, and set 
$$p^c = (n - \#_1(p), \ldots, n - \#_n(p)).$$ 
Notice $0 \leq (p^c)_i \leq m$ for each $i = 1, \ldots, n$. For any tuple $z$ in $\mathbb{Z}_+^n$, let $2z$ denote $(2z_1, \ldots, 2z_n)$. Now define
\begin{align}\label{prime part hwv}
	\xi_p = v_{2 p^c}^* w^p.
\end{align}
\begin{figure}[!h]
	\begin{center}
	\begin{tikzpicture}
	\node at (0, \yy) 
	{\begin{varwidth}{2.5cm}{
		\begin{ytableau}
			*(green!20) &  *(green!20) \\ 
			*(green!20) & *(blue!60) \\
			*(blue!30) & *(blue!60) \\
		\end{ytableau}}\end{varwidth}};
	\node at (0, 0) {$\mu = (2, 1)$};
	\node at (5, \yy) {\begin{varwidth}{5cm}{
		\begin{ytableau}
			*(green!20) v_{11} 
			& *(green!20) v_{12} 
			& *(green!20) v_{13} 
			& *(green!20) v_{14} \\
			*(blue!60) v_{21} 
			&  
			& *(green!20) v_{23} 
			& *(green!20) v_{24} \\
			*(blue!60) v_{31} 
			& 
			& *(blue!30) v_{33} 
			& \\ 
		\end{ytableau}}\end{varwidth}};
	\node (C) at (1.5, \yy) {};
	\node (D) at (3, \yy) {};
	\draw[->,semithick] (C) edge (D);
	\end{tikzpicture}
	\end{center}
	\caption{The element $\xi_\mu = v_{2\mu^c}^* w^\mu$ of $\bigwedge_q(V^{(nm)})$ represented by states of occupied and vacant positions in an $n \times m$ grid, with $(n, m) = (3, 4)$, $\mu = (2, 1)$, and $p = (2, 1)$.}
	\label{orthog hwv}
\end{figure}
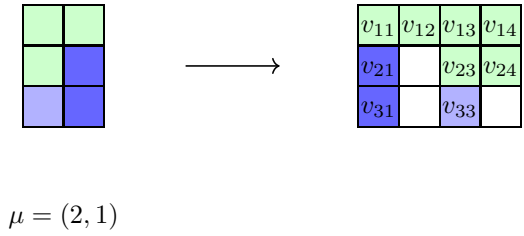

\begin{figure}[!h]
\begin{center}
\begin{tikzpicture}
\node at (0, \yy) {\begin{varwidth}{5cm}{
	\begin{ytableau}
		*(green!20) v_{11} 
		& *(green!20) v_{12} 
		& *(green!20) v_{13} 
		& *(green!20) v_{14} \\
		*(blue!60) v_{21} 
		&  
		& *(green!20) v_{23} 
		& *(green!20) v_{24} \\
		*(blue!60) v_{31} 
		& 
		& *(blue!30) v_{33} 
		&  \\
	\end{ytableau}}\end{varwidth}};
\node (A) at (1.75, \yy) {};
\node (B) at (4, \yy) {};
\node at (2.875, \yy + 0.5) {
	$\mathcal{R}_q\big(F_1^{(m)}\big)$
};
\draw[->,semithick] (A) edge (B);
\node at (4.55, \yy) {$\alpha_q$};
\node at (6.2, \yy) 
{\begin{varwidth}{5cm}{
	\begin{ytableau}
		*(green!20) v_{11} 
		& *(green!20) v_{12} 
		& *(green!20) v_{13} 
		& *(green!20) v_{14} \\ 
		& *(blue!60) v_{22}
		& *(green!20) v_{23} 
		& *(green!20) v_{24} \\
		*(blue!60) v_{31} 
		& 
		& *(blue!30) v_{33} 
		&  \\
	\end{ytableau}}\end{varwidth}};
\node at (8, \yy) {$+$};
\node at (8.7, \yy) {$\alpha'_q$};
\node at (10.4, \yy) 
{\begin{varwidth}{5cm}{
	\begin{ytableau}
		*(green!20) v_{11} 
		& *(green!20) v_{12} 
		& *(green!20) v_{13} 
		& *(green!20) v_{14} \\ 
		*(blue!60) v_{21}
		& 
		& *(green!20) v_{23} 
		& *(green!20) v_{24} \\ 
		& *(blue!60) v_{32}
		& *(blue!30) v_{33} 
		&  \\ 
	\end{ytableau}}\end{varwidth}};
\node (A) at (1.75, -0.5*\yy) {};
\node (B) at (4, -0.5*\yy) {};
\node at (2.875, -0.5*\yy + 0.5) {
	$\mathcal{R}_q\big(F_1^{(m)}\big)$
};
\draw[->,semithick] (A) edge (B);
\node at (4.55, -0.5*\yy) {$\beta_q$};
\node at (6.2, -0.5*\yy) 
{\begin{varwidth}{5cm}{
	\begin{ytableau}
		*(green!20) v_{11} 
		& *(green!20) v_{12} 
		& *(green!20) v_{13} 
		& *(green!20) v_{14} \\ 
		& *(blue!60) v_{22}
		& *(green!20) v_{23} 
		& *(green!20) v_{24} \\ 
		& *(blue!60) v_{32}
		& *(blue!30) v_{33} 
		&  \\ 
	\end{ytableau}}\end{varwidth}};
\node (A) at (1.75, -2*\yy) {};
\node (B) at (4, -2*\yy) {};
\node at (2.875, -2*\yy + 0.5) {
	$\mathcal{R}_q\big(F_1^{(m)}\big)$
};
\draw[->,semithick] (A) edge (B);
\node at (4.55, -2*\yy) {$0$};
\end{tikzpicture}
\end{center}
\caption{Shifting occupied positions in the $1$st pair of columns to the right by applying $F_{2(1)-1}^{(m)} \in \Uqglm$ to the $\Uqdn$ weight vector $\xi_p$, with $p = (2, 1)$ and $(n, m) = (3, 4)$. Here $\alpha_q, \alpha_q'$, and $\beta_q$ are some coefficients.}
\label{f on prod ex}
\end{figure}

\begin{figure}[!h]
\begin{center}
\begin{tikzpicture}
\node at (0, \yy) {\begin{varwidth}{5cm}{
	\begin{ytableau}
		*(green!20) v_{11} 
		& *(green!20) v_{12} 
		& *(green!20) v_{13} 
		& *(green!20) v_{14} \\
		*(blue!60) v_{21} 
		&  
		& *(green!20) v_{23} 
		& *(green!20) v_{24} \\
		*(blue!60) v_{31} 
		& 
		& *(blue!30) v_{33} 
		&  \\
	\end{ytableau}}\end{varwidth}};
\node (A) at (1.75, \yy) {};
\node (B) at (4, \yy) {};
\node at (2.875, \yy + 0.5) {
	$\mathcal{R}_q\big(F_3^{(m)}\big)$
};
\draw[->,semithick] (A) edge (B);
\node at (4.45, \yy) {$\gamma_q$};
\node at (6.2, \yy) 
{\begin{varwidth}{5cm}{
	\begin{ytableau}
		*(green!20) v_{11} 
		& *(green!20) v_{12} 
		& *(green!20) v_{13} 
		& *(green!20) v_{14} \\
		*(blue!60) v_{21} 
		&  
		& *(green!20) v_{23} 
		& *(green!20) v_{24} \\
		*(blue!60) v_{31} 
		& 
		&   
		& *(blue!30) v_{34} \\
	\end{ytableau}}\end{varwidth}};
\node (A) at (7.65, \yy) {};
\node (B) at (9.9, \yy) {};
\node at (8.775, \yy + 0.5) {
	$\mathcal{R}_q\big(F_3^{(m)}\big)$
};
\draw[->,semithick] (A) edge (B);
\node at (10.4, \yy) {0};
\end{tikzpicture}
\caption{Shifting occupied positions in the $2$nd pair of columns to the right by applying $F_{2(2) - 1}^{(m)} \in \Uqglm$ to the $\Uqdn$ weight vector $\xi_p$, with $p = (2, 1)$ and $(n, m) = (3, 4)$. Here $\gamma_q$ is some coefficient.}
\label{e on prod ex}
\end{center}
\end{figure}

The next lemma records useful properties of $\xi^p$. For any partition $\mu$, we use $\mu'$ to denote the conjugate of $\mu$.
\begin{lem}\label[lem]{q wp props}
	For any $p = (p_1, \ldots, p_{r})$ with $p_j \leq n$, let $\xi^p$ denote the vector defined by \Cref{prime part hwv}. Then $\xi^p$ is
	\begin{enumerate}[(1)]
		\item always nonzero.
		\item a weight vector with respect to the action of $\Uqdn$ with weight $\mu$ if $p_j = n - \mu_j'$ for some partition $\mu$ with at most $n$ parts satisfying $\mu_1 \leq r$.
		\item a highest weight vector in the $(p_j + 1)$-dimensional module $U_q(\mathfrak{sl}_2) \rhd \xi^p$ of the copy of $U_q(\mathfrak{sl}_2) \subset \Uqglm$ generated by $E_{2j-1}^{(m)}, F_{2j-1}^{(m)}, K_{2j-1}^{(m)}$. 
		\item a highest weight vector with respect to the action of $\Uqgln \subset \Uqdn$. In particular, it is annihilated by $\mathcal{L}_q(E_i^{(n)})$, for $i = 1, \ldots, n-1$.
	\end{enumerate}
\end{lem}
\begin{proof}
	To see that $\xi^p$ is nonzero, notice that $v_{2 p^c}^*$ and $w^p$ do not share common factors. \Cref{uqdn weight on basis} determines the weight of $\xi_p$. When $p_j = n - \mu_j'$, each position in the $n \times m$ grid is occupied by $v_{2\mu}$, occupied by $w^p$, or it is paired bijectively to a box occupied by $w^p$. Since each occupied (vacant) position in the $i$th row contributes $+1/2$ ($-1/2$) to the $i$th component of the weight, the positions occupied by $v_{2\mu}^*$ contribute $\mu$ while boxes occupied by $w^p$ have no overall contribution to the weight.
	 
	\Cref{e on prod ex,f on prod ex} sketch the proof of $(3)$. The general case follows from observing that the formulas defining $\mathcal{R}_q$ in \Cref{uqdn clqnm embedding} imply that, e.g., when $p_j \geq 0$, each application of $\mathcal{R}_q(F_{2j-1}^{(m)})$ on $\xi^p$ moves one occupied position in the $(2j-1)$st column below the $\mu_j'$th row to the right, and a further application of $\mathcal{R}_q(F_{2j-1}^{(m)})$ when all the occupied positions below the $\mu_j'$th row are in the $(2j)$th column kills the vector. An analogous story holds for $\mathcal{R}_q(E_{2j-1}^{(m)})$ when $p_j \leq 0$.
	
	Statement $(4)$ follows immediately from the definition of $\mathcal{L}_q\colon \Uqgln \to Cl_q(nm)$ in \Cref{uqgln clqnm embedding}, since every box above an occupied box in $\xi^p$ is occupied. 
\end{proof}

We are now ready to construct the $\Uqprime$ weight vectors in $S^{\otimes m}$. 
\begin{lem}\label[lem]{uqprime wv}
Let $m = 2r$ or $m = 2r + 1$. Consider any $p = (p_1, \ldots, p_{r})$, with $-n \leq p_j \leq n$, for each $j = 1, \ldots, r$. Define $\mathrm{abs}(p) = (|p_1|, \ldots, |p_{r}|)$ and let $s = (s_1, \ldots, s_{r})$, with $s_j = \mathrm{sgn}(p_j) \in \{\pm 1\}$, denote the sign of $p_j$. Set $\mathcal{I}_{p_j} = \{0, 1, 2, \ldots, |p_j|\}$, and let $\mathcal{I}^p = \mathcal{I}_{p_1} \times \cdots \times \mathcal{I}_{p_{r}}$. For each $a \in \mathcal{I}$, let $\theta(a) = -{1 \over 2}\sum_j (|p_j| - a_j) (|p_j| - a_j - 1)$, and set $[a]! = [a_1]! \cdots [a_{r}]!$ and $F^a = \prod_{j=1}^{r} \mathcal{R}_q(F_{2j-1})^{a_j}$. Now define
\begin{align}\label{xip def}
	\Xi_p = \sum_{a \in \mathcal{I}^p} i^{s \cdot a} q^{\theta(a)} {F^a \over [a]!} \xi_{\mathrm{abs}(p)}.
\end{align}
Here $i = \sqrt{-1}$ and $s \cdot a$ denotes the usual dot product of $s$ and $a$. 

Then $\Xi_p$ is a $\Uqprime$ weight vector of weight $p$ when $m = 2r$ is even. When $m = 2r + 1$ is \textit{odd}, multiply $\Xi_p$ by $v(\sum_{i=1}^n e_{i + (j-1)n})$ to fill its last column and obtain a $\Uqprime$ weight vector of weight $p + (1/2, \ldots, 1/2, +1/2)$. Similarly, multiply $\Xi_p$ by $v(\sum_{i=1}^{n-1} e_{i + (j-1)n})$ to obtain a weight vector of weight $p + (1/2, \ldots, 1/2, -1/2)$.
\end{lem}
\begin{proof}
We show that $\Xi_p$ is a $[p_j]$-eigenvector of each $B_{2j-1}$. This boils down to a calculation in the copy of $U_q(\mathfrak{sl}_2) \subset \Uqglm$ generated by $E_{2j-1}, F_{2j-1}, K_{2j-1}$, since $[E_{2j-1}, F_{2k-1}] = [F_{2j-1}, F_{2k-1}] = 0$ whenever $j \neq k$. In what follows we only consider the $j$th inner sum defining $\Xi_p$. 

Applying $F_{2j-1}$ merely shifts the index of summation:
\begin{align*}
	F_{2j-1} &\rhd 
		\sum_{b=0}^{|p_j|} 
		  {i^{s_j b} \over [b]!} 
		  q^{\theta(b)} 
		  \mathcal{R}_q(F_{2j-1})^{b} \xi_{\mathrm{abs}(p)} \\
		 &= 
		  i^{-s_j} \sum_{b=0}^{|p_j|-1} 
		  {i^{s_j (b+1)} \over [b]!} 
		  q^{\theta(b+1) - \left(|p_j| - (b+1)\right)} 
		  \mathcal{R}_q(F_{2j-1})^{b+1} \xi_{\mathrm{abs}(p)} \\
		 &= 
		  i^{-s_j} \sum_{b=1}^{|p_j|} 
		  {i^{s_j b} \over [b]!} 
		  q^{\theta(b)} 
		  \mathcal{R}_q(F_{2j-1})^{b} \xi_{\mathrm{abs}(p)}. 
\end{align*}
In the second equality we used the identity $\theta(b) = -{1 \over 2}(|p_j| - b - 1)(|p_j| - b - 2)  - |p_j| + b - 1 = \theta(b+1) - \left(|p_j| - (b+1)\right)$. The term with $F_{2j-1}^{|p_j|+1}$ vanishes because $\xi_{\mathrm{abs}(p)}$ is a highest weight vector in a $(|p_j|+1)$-dimensional $U_q(\mathfrak{sl}_2)$-module.

Next we calculate
\begin{align*}
	q K_{2j-1}^{-1} E_{2j-1} &\rhd 
		\sum_{b=0}^{|p_j|} 
		  {i^{s_j b} \over [b]!} 
		  q^{\theta(b)} 
		  \mathcal{R}_q(F_{2j-1})^{b} \xi_{\mathrm{abs}(p)} \\
		 &= i^{s_j} \sum_{b=1}^{|p_j|} 
		  {i^{s_j (b-1)} \over [b]!} 
		  q^{\theta(b)}
		  \mathcal{R}_q\left(K_{2j-1}^{-1}\left(F_{2j-1}^b E_{2j-1} + [E_{2j-1}, F_{2j-1}^b]\right)\right) \xi_{\mathrm{abs}(p)} \\
		 &= i^{s_j} \sum_{b=1}^{|p_j|} 
		  {i^{s_j (b-1)} \over [b-1]!} 
		  q^{\theta(b-1) + \left(|p_j| - (b-1)\right) + (2(b-1) - |p_j|)}
		  [|p_j| - b]
		  \mathcal{R}_q(F_{2j-1})^{b-1} \xi_{\mathrm{abs}(p)} \\
		 &= i^{s_j} \sum_{b=0}^{|p_j| - 1} 
		  {i^{s_j b} \over [b]!} 
		  q^{\theta(b) + b}
		  [|p_j| - b]
		  \mathcal{R}_q(F_{2j-1})^{b} \xi_{\mathrm{abs}(p)}.
\end{align*}
The second equality follows from the $U_q(\mathfrak{sl}_2)$ identity 
$$[E_{2j-1}, F_{2j-1}^b] = [b] F_{2j-1}^{b-1} {q^{-(b-1)} K_{2j-1} - q^{(b-1)} K_{2j-1}^{-1} \over q - \qinv}$$ 
that is proved in, e.g., Lemma~VI.1.3 of \cite{Kassel}, and the easy identity $\theta(b) = -{1 \over 2} \left(|p_j| - (b-1)\right)\left(|p_j| - b\right) +  \left(|p_j| - (b-1)\right) - 1 = \theta(b-1) + \left(|p_j| - (b-1)\right) - 1$.

Let $\widehat{\mathcal{I}}^p_j = \mathcal{I}_{p_1} \times \cdots \times \widehat{\mathcal{I}}_{p_j} \times \cdots \times \mathcal{I}_{p_{r}}$ denote the set obtained by omitting $\mathcal{I}_{p_j}$ from the product. Taking the $b = 0$ term from the second calculation, the $b = |p_j|$ term from the first, and adding the rest, we find that $\Xi_p$ is indeed a $[p_j]$-eigenvector of $B_{2j-1}$:
\begin{align*}
	B_{2j-1} \rhd \Xi_p 
	&= i\mathcal{R}_q(F_{2j-1} - q K_{2j-1}^{-1} E_{2j-1}) \rhd \Xi_p \\
	&= i^{1 - s_j} 
	   \sum_{a \in \widehat{\mathcal{I}}^p_j} i^{s \cdot a} q^{\theta(a)} 
	   {F^a \over [a]!}
	   \left(
	   		q^{-{1 \over 2}|p_j|(|p_j| - 1)}[|p_j|] 
	   		\right. \\
	   	  &\qquad
	   	  + \sum_{b=1}^{|p_j| - 1} 
	   	  	{i^{s_j b} \over [b]!}
	   	  	q^{\theta(b)}
	   	  	\left(q^b [|p_j| - b] + q^{-(|p_j| - b)} [b]\right)
	   	  	\mathcal{R}_q(F_{2j-1})^b \\
	   	  &\qquad \left.
	   	  + i^{s_j |p_j|} {\mathcal{R}_q(F_{2j-1})^{|p_j|} \over [|p_j|]!}
	   	  	[|p_j|] 
	   	\right)
	   	\xi_{\mathrm{abs}(p)} \\
	  &= i^{1 - s_j} [|p_j|] \sum_{a \in \mathcal{I}} i^{s \cdot a} q^{\theta(a)} 
	  	 {F^a \over [a]!} \xi_{\mathrm{abs}(p)} \\
	  &= [p_j] \Xi_p.
\end{align*}
The third equality follows from the identity of $q$-integers $q^n [m] + q^{-m} [n] = [n + m]$.
\end{proof}

Finally, we obtain joint highest weight vectors in $S^{\otimes m} \cong \bigwedge_q(V^{(nm)})$ with respect to the $\Uqdn \otimes \Uqprime$ action.
\begin{lem}\label[lem]{xip is joint hwv}
	Let $\mu$ denote a partition with at most $n$ parts satisfying $\mu_1 \leq r$ and take $\bar{\mu}$ as in \Cref{mu bar} and $\Xi_p$ as in \Cref{xip def}. Then $\Xi_{\bar{\mu}}$ is a $\Uqdn$ highest weight vector of weight $\mu$.
\end{lem}
\begin{proof}
\Cref{q wp props} shows that $\xi_{\bar{\mu}}$ has weight $\mu$ with respect to the $\Uqdn$ action and that $\mathcal{L}_q(E_i^{(2n)}) \xi_{\bar{\mu}} = 0$ for each $i < n$. Since $\mathcal{L}_q$ and $\mathcal{R}_q$ induce commuting module endomorphisms by \Cref{uqdn uqprime commute}, $\Xi_{\bar{\mu}}$ also has weight $\mu$ with respect to the $\Uqdn$ action and $\mathcal{L}_q(E_i^{(n)}) \Xi_{\bar{\mu}} = 0$ for each $i < n$. 

It remains to show $\Xi_{\bar{\mu}}$ is also annihilated by $E_n \in \Uqdn$. This is clear when $\bar{\mu}_1 \leq 1$, since every position in the $(n-1)$st row of $\xi_{\bar{\mu}}$ is already occupied. Let $\eta = \mathcal{L}_q(E_n^{(2n)}) \xi_{\bar{\mu}}$ and suppose ${\bar{\mu}}_1 \geq 2$. Then $\mathcal{R}_q(F_{2j-1})^{\bar{\mu}_1 - 1} \eta = 0$, so $\eta$ now generates $U_q(\mathfrak{sl}_2)$-module of dimension ${\bar{\mu}}_1 - 1 < {\bar{\mu}}_1 + 1$. Since $\mathcal{L}_q$ and $\mathcal{R}_q$ induce commuting module endomorphisms, $\eta$ must also be $\Uqprime$ weight vector of weight $\bar{\mu}$. But this contradicts the eigenvalue calculation of \Cref{uqprime wv} unless $\eta$ is zero.
\end{proof}
\begin{proof}[(Proof of \Cref{uqodn uqprime duality})]
	Combined with \Cref{xip def}, \Cref{xip is joint hwv} proves that $\Xi_{\bar{\mu}}$ is a joint highest weight vector for the $\Uqdn \otimes \Uqprime$ action on $S^{\otimes m}$ for each partition $\mu$ with at most $n$ parts satisfying $\mu_1 \leq r$. 
	
	Now consider partitions $\mu$ with at most $2n$ parts satisfying $\mu_1' + \mu_2' \leq 2n$. If $\mu \neq \mu^\dagger$, then we obtain two linearly independent joint highest weight vectors $\Xi_{\bar{\mu}}$ and $\Xi_{\bar{\mu}^\dagger}$ with the same $\Uqdn$ weight, but distinct $\Uqprime$ weight. However, these generate inequivalent $\Uqodn$-modules. Conversely if $\mu = \mu^\dagger$ then there is a unique irreducible $\Uqson$-module that extends to an irreducible $\Uqodn$-module.
	
	Hence, we have exhibited a joint highest weight vector for each irreducible module appearing in the decomposition \Cref{uqodn uqprime duality decomp}. Since the dimension of each $\Uqodn$- and each $\Uqprime$-module equals that of its classical counterpart, the theorem follows\longer{ exactly like our skew duality \Cref{uqgln uqglm duality} for Type $\mathbf{A}$}: a dimension count together with the classical duality result guarantee that we have exhausted every possible irreducible component. 
\end{proof}

As in the classical case, we obtain a decomposition of $\bigwedge_q(V^{(nm)})$ as a $U_q'(\mathfrak{so}_n) \otimes U_q(\mathfrak{so}_{2m})$-modules by considering the isomorphism of $U_q(\mathfrak{so}_{2m})$-modules $S_m^{\otimes n} \cong \bigwedge_q(V^{(m)})^{\otimes n} \cong \bigwedge_q(V^{(nm)})$. This decomposition completes the seesaw of \Cref{quantum seesaw}.


\bibliographystyle{alphaurl}
\bibliography{ref}
\end{document}